\documentclass[12pt,final,a4paper,reqno,psamsfonts]{amsart}
\usepackage[latin1]{inputenc}
\usepackage{amssymb,amscd,amsxtra,verbatim,ifthen,graphicx}
\usepackage{showlabels} 
\usepackage[english]{babel}
\usepackage[all]{xy}
\usepackage{hyperref}
\usepackage{enumitem}

\mathchardef\ordinarycolon\mathcode`\: \mathcode`\:=\string"8000
\begingroup \catcode`\:=\active
  \gdef:{\mathrel{\mathop\ordinarycolon}}
\endgroup

\theoremstyle{plain}
\newtheorem{theorem}{Theorem}[section]
\newtheorem{proposition}[theorem]{Proposition}

\newtheorem{lemma}[theorem]{Lemma}

\newtheorem{corollary}[theorem]{Corollary}

\theoremstyle{definition}
\newtheorem{definition}[theorem]{Definition}

\theoremstyle{remark}
\newtheorem{remark}[theorem]{Remark}

\newtheorem{example}[theorem]{Example}

\newcommand{\im}{\mathrm{im}}
\newcommand{\pr}{\mathrm{pr}}

\newcommand{\Fix}{\mathrm{Fix}}

\newcommand{\reg}{\mathrm{reg}}

\newcommand{\princ}{\mathrm{princ}}

\newcommand{\Ad}{\mathrm{Ad}}
\newcommand{\ad}{\mathrm{ad}}
\newcommand{\copol}{\mathrm{copol}}
\newcommand{\cohom}{\mathrm{cohom}}
\newcommand{\covar}{\mathrm{covar}}

\newcommand{\hor}{\mathrm{hor}}
\newcommand{\basic}{\mathrm{basic}}

\renewcommand{\span}{\mathrm{span}}

\newcommand{\Gln}{\mathbf{GL}}

\newcommand{\On}{\mathbf{O}}

\newcommand{\Son}{\mathbf{SO}}

\newcommand{\Un}{\mathbf{U}}

\newcommand{\Sun}{\mathbf{SU}}

\newcommand{\iso}{\mathrm{Iso}}

\newcommand{\R}{\mathbf{R}}
\newcommand{\C}{\mathbf{C}}

\newcommand{\Z}{\mathbf{Z}}
\newcommand{\N}{\mathbf{N}}

\newdir^{ (}{{}*!/-5pt/@^{(}}
\newcommand{\kc}{{\mathcal C}}
\newcommand{\kd}{{\mathcal D}}
\newcommand{\ke}{{\mathcal E}}
\newcommand{\kf}{{\mathcal F}}

\newcommand{\kh}{{\mathcal H}}

\newcommand{\kj}{{\mathcal J}}

\newcommand{\km}{{\mathcal M}}

\newcommand{\kp}{{\mathcal P}}

\newcommand{\ks}{{\mathcal S}}

\newcommand{\kv}{{\mathcal V}}

\newcommand{\gothg}{{\mathfrak g}}
\newcommand{\gothh}{{\mathfrak h}}

\newcommand{\gothk}{{\mathfrak k}}
\newcommand{\gothl}{{\mathfrak l}}
\newcommand{\gothm}{{\mathfrak m}}
\newcommand{\gothn}{{\mathfrak n}}

\newcommand{\gothp}{{\mathfrak p}}

\newcommand{\goths}{{\mathfrak s}}

\newcommand{\gothz}{{\mathfrak z}}

\begin{document}
\title{Reductions, Resolutions and the Copolarity of Isometric Group Actions}
\author{Frederick Magata}
\maketitle

\begin{abstract}We present some results on reductions and the copolarity of isometric group actions, which we obtained in our thesis \cite{Mag1}. We also describe a resolution construction for isometric actions with respect to a reduction and give examples.
\end{abstract}

\selectlanguage{english}

\section{Introduction}
A \emph{reduction} of an isometric action $(G,M)$ consists of a \emph{fat section} $\Sigma$ and a \emph{fat Weyl group} $W$ acting on $\Sigma$.
\emph{Fat section} were first introduced, under a different name, in \cite{GOT}. The motivation for this comes from \emph{polar actions} and their \emph{sections}. In this situation $\Sigma$ is a complete, connected and embedded submanifold which intersects every $G$-orbit and is perpendicular to them in every intersection point. Such actions have many fascinating properties and are well studied in the literature (see for instance \cite{Dad,PT2,BCO,Kol} for key results and further references). In \cite{GOT} Gorodski, Olmos and Tojeiro tried to measure the defect of an arbitrary isometric action from being polar. This led them to the notion of a fat section and the integer valued invariant \emph{copolarity} of an isometric action. In this picture, polar actions are precisely the copolarity-$0$ actions. The authors obtained a classification of all irreducible copolarity-$1$ representations which in turn enabled them to characterize all irreducible orthogonal \emph{taut} representations as those of copolarity $0$ or $1$.

Our paper is organized as follows: In Section \ref{s:fatdef} we define fat sections, fat Weyl groups, reductions and the copolarity of isometric actions and show some basic properties. In Section \ref{s:properties} the main result is that the orbit space of any reduction is isometric to the orbit space of the original action. Together with the implications obtained from this result this serves as a justification for our definition of a reduction. 

In Section \ref{s:slicerep} we show that the copolarity of the slice representation in any point cannot exceed that of the global action. As consequences of this we show in Section \ref{s:stability} that a point on a fat section is $G$-regular if and only if it is $W$-regular with respect to the corresponding reduction, and that the copolarity does not change if we pass from an isometric action to any of its reductions. 

In Section \ref{s:varcomp} we show that reductions also behave well with respect to \emph{variational completeness}: an isometric action is variationally complete if and only if some/every reduction is variationally complete. At the end of this section we also generalize results on \emph{variational co-completeness} from \cite{GOT}.

In Section \ref{s:resolution} we generalize the \emph{resolution construction} from \cite{GS} using arbitrary fat sections. As an application one can construct from any $G$-space $M$ with sectional curvature bounded from below by some constant $\kappa\le0$ another $G$-space $\tilde M$ with the same curvature bound from below and such that both spaces have isomorphic orbit spaces.

Section \ref{s:chevalley} deals with Chevalley's restriction theorem for reductions, which we are able to show under additional assumptions. A class of isometric actions on compact Lie groups associated with Riemannian symmetric spaces where these conditions are met is explained in Section \ref{s:symspace}. 

In Section \ref{s:infinite} we investigate a class of affine isometric actions on Hilbert spaces related with the examples of Section \ref{s:symspace}. Using an adapted notion of copolarity, we are able to show an interesting dichotomy: Depending on whether the base action is hyperpolar or just polar, the copolarity of the infinite dimensional action is either $0$ or $\infty$.

Finally, Section \ref{s:srf} describes how fat sections can be defined for \emph{singular Riemannian foliations}.

\section[Fat Sections, Fat Weyl Groups and the Copolarity]{Fat Sections, Fat Weyl Groups and the Copolarity of Isometric Actions}\label{s:fatdef}
By an \emph{isometric action} of a Lie group $G$ on a Riemannian manifold $M$ we mean a smooth and proper
homomorphism $\Phi:G\to\iso(M)$. An action is also denoted by the associated
map $\varphi:G\times M\to M, (g,q)\mapsto g\cdot q:=\Phi(g)(q)$, or
just by $(G,M)$. \emph{$G$-Regular points} are points lying on principal orbits. Their collection is denoted by $M^\reg$. All other points are
called singular. Thus, points lying on exceptional orbits are also singular in our sense.

\begin{definition}\label{d:copolar}
Let $M$ be a complete Riemannian manifold and let $(G,M)$ be an isometric action.
A submanifold $\Sigma\subseteq M$ is called a \textbf{fat section} of $(G,M)$ if:
\begin{enumerate}[label=(\Alph{*}), ref=(\Alph{*})]
\item\label{pr:1} $\Sigma$ is complete, connected, embedded and totally
geodesic in $M$,

\item\label{pr:2} $\Sigma$ intersects every orbit of the $G$-action,

\item\label{pr:3} for all $G$-regular $p\in \Sigma$ we have $\nu_p(G\cdot
p)\subseteq T_p\Sigma$,

\item\label{pr:4} for all $G$-regular $p\in \Sigma$ and $g\in G$ such that $g\cdot
p\in \Sigma$ we have $g\cdot \Sigma=\Sigma$.
\end{enumerate}
In this situation, following \cite{GOT}, we also call $\Sigma$ a \textbf{$k$-section},
where $k$ denotes the codimension  of $\nu_p(G\cdot p)$ in
$T_p\Sigma$ for any regular point $p\in\Sigma$. The integer
$$\copol(G,M):=\min\{k\in \textbf{N}\mid \text{there is a } k\text{-section } \Sigma\subseteq M\}$$
is called the \textbf{copolarity} of the $G$-action on $M$. If $\Sigma\subseteq M$
is a $\copol(G,M)$-section, then we say that $\Sigma$ is
\textbf{minimal}. If a submanifold $\Sigma\subseteq M$ satisfies only
properties \ref{pr:1}-\ref{pr:3} above, it is called
\textbf{pre-section}. Finally, if $M$ is a minimal section of $(G,M)$,
we say that $(G,M)$ has \textbf{trivial copolarity}.
\end{definition}
\begin{remark}{\ }
\begin{enumerate}
\item An isometric action $(G,M)$ is called \emph{polar} if there exists a complete, connected and embedded submanifold $\Sigma$, called \emph{section}, which intersects every orbit and such that in the intersection points the orbits are perpendicular to $\Sigma$. It follows that such a $\Sigma$ is totally geodesic and satisfies property \ref{pr:4} in the above definition. Hence, $\copol(G,M)=0$ and a section in the polar sense is a (minimal) $0$-section in the sense of Definition \ref{d:copolar}. Conversely, an isometric action with copolarity zero is polar and all minimal sections are sections in the polar sense. The copolarity therefore measures the failure of an isometric action to be polar.
\item For a given Riemannian manifold $M$, one can define the \textbf{copolarity of $M$} as:
$$\copol(M):=\copol(\iso(M),M).$$
Just like the symmetry rank, symmetry degree and the cohomogeneity of a Riemannian manifold (see for instance \cite{Wil2} for the definitions), the copolarity is also a measure for the amount of symmetry a Riemannian manifold carries. 
For instance, homogeneous spaces and cohomogeneity one manifolds are manifolds of copolarity zero.
\end{enumerate}
\end{remark}

Situations in which the copolarity of an action is nontrivial and not equal to zero and where the minimal sections can be explicitly computed are described in Section \ref{s:symspace} and \cite{Mag1,Mag2}. To give some flavor:
\begin{example}\label{e:standard} The $k$-fold direct sum of the standard representation of $\Son(n)$ on $\R^n$ has nontrivial copolarity equal to $\frac{k(k-1)}{2}$ for $2\le k\le n-1$ and a minimal section is given by $\R^{k^2}$, which is embedded into $\R^{kn}$ as block matrices with nonzero entries in the upper $(k\times k)$-block only.
\end{example}

\begin{example}
Consider the following action of $T^2\!\times\! S(\Un(1)\!\times\!\Un(2))$ on $\Sun(3)$. The first factor acts by matrix multiplication from the left and the second factor by matrix multiplication from the right by the inverted matrix. The copolarity in this case is equal to $1$ and a minimal section is given by $\Son(3)\subset\Sun(3)$.
\end{example}

The following three lemmas are frequently used throughout the paper.
Lemma \ref{l:sliceequiv} and \ref{l:sliceregular} are \cite[Lemma 5.1 and Lemma 5.2]{GOT}. Originally, the second of these was stated for orthogonal representations only, but its proof also works in the general case.
\begin{lemma}\label{l:slice_intersect}
Let $(G,M)$ be an isometric action and suppose that $M$ is connected and finite dimensional. If $p\in M^\reg$, then $\exp_p(\nu_p(G\cdot p))$ intersects every $G$-orbit.
\end{lemma}

\begin{lemma}\label{l:sliceequiv}
Let $(G,M)$ be an isometric action and let $q\in M$ be arbitrary. For $v\in\nu_q(G\cdot q)$ the following assertions are equivalent:
\begin{enumerate}
\item $v$ is $G_q$-regular.

\item There exists $\varepsilon>0$ such that $\exp_q(tv)$ is
$G$-regular for $0<t<\varepsilon$.

\item $\exp_q(t_0v)$ is $G$-regular for some $t_0>0$.
\end{enumerate}
\end{lemma}

\begin{lemma}\label{l:sliceregular}
Let $\Sigma$ be a fat section of $(G,M)$. For all $q\in\Sigma$ there is a $G_q$-regular $v\in T_q\Sigma\cap\nu_q(G\cdot q)$. Furthermore, $v$
can be chosen such that $p=\exp_q v$ is $G$-regular and arbitrarily close to $q$.
\end{lemma}


The following proposition lists several properties related to the copolarity of an isometric action. They are either observations already made in \cite{GOT} or immediate consequences of them and Definition \ref{d:copolar}.
\begin{proposition}\label{p:copolarproperties}
Let $M,N$ be finite dimensional Riemannian manifolds and $G,H$ Lie groups which act
smoothly and isometrically on $M$, resp. $N$. Let furthermore
$p\in M$ be an arbitrary $G$-regular point.
\begin{enumerate}
\item If $(G,M)$ and $(H,N)$ are orbit-equivalent (i.e. there is an isometry from $M$ onto $N$, mapping $G$-orbits onto $H$-orbits), then
$\copol(G,M)=\copol(H,N)$.

\item $\copol(G,M)=\copol(G^\circ,M)$, where $G^\circ$ denotes the identity
component.

\item For any two fat sections $\Sigma_1, \Sigma_2$ containing
$p$, the \textbf{connected intersection} (i.e. the connected
component of $p$ of the intersection $\Sigma_1\cap\Sigma_2$) is again a fat
section. Hence, a minimal section through $p$ is unique.

\item The minimal section through $p$ is the connected
intersection of all fat sections containing $p$, and also the
connected intersection of all pre-sections through $p$.

\item The $G$-translates of a fat section $\Sigma$ foliate $M^\reg$, the $G$-regular points of $M$.

\item $G$ is transitive on the set of all minimal sections of $(G,M)$.

\item The intersection of a principal orbit $G\cdot p$ with a fat section $\Sigma$ is
an embedded submanifold of $M$. Moreover, if $N_G(\Sigma)$ denotes the normalizer of $\Sigma$ in $G$, then
$$\Sigma\cap(G\cdot p)=N_G(\Sigma)\cdot p, \text{ if } p\in\Sigma.$$
\item If $\Sigma$ is a fat section, then $\Sigma^\reg:=\Sigma\cap M^\reg$ is open and dense in $\Sigma$.
\end{enumerate}
\end{proposition}

Clearly, $M$ itself is always a fat section of $(G,M)$ (hence, we speak of trivial copolarity if $M$ is the only fat section). More interesting fat sections can often be found using
\begin{proposition}[{\cite[Section 3.2]{GOT}}]\label{p:fixpointset}
If $(G,M)$ is isometric and $p\in M^\reg$, then $\Sigma:=\Fix(G_p, M)^\circ$, i.e. the connected component of $p$ of the fixed point set of $G_p$ is a $k$-section, where $k=\dim(T_p(G\cdot p)^{G_p})$. 
\end{proposition}
\begin{definition}\label{d:can_fat}
We call a fat section as in Proposition \ref{p:fixpointset} \textbf{canonical section}. Furthermore, we say a fat section is \textbf{sufficiently small} if it is contained in some canonical section. In particular, canonical sections and minimal sections are sufficiently small.
\end{definition}
\begin{remark}[\cite{GOT}, Section 3.2]{\ }
Canonical sections need not be minimal sections. For instance, for $k=2, n=3$ in Example \ref{e:standard} the principal isotropy groups are trivial, but minimal sections are proper subspaces of the representation space. Nevertheless, for an isometric action $(G,M)$ the acting group $G$ can often be enlarged to a group $G'$, which also acts isometrically on $M$ with the same orbits as $G$, and such that $(G',M)$ has canonical minimal sections. By Proposition \ref{p:copolarproperties} (ii) both actions have the same copolarity and minimal sections. It is interesting to note that for every polar representation the sections can be obtained in this way (\cite[Theorem 1.3]{S}), and this is also the case for the representations in Example \ref{e:standard} (see also \cite[Chapter 7]{Mag1}).
\end{remark}

\begin{definition}
Let $\Sigma$ be a fat section of the isometric action $(G,M)$. We
put
$$W=W(\Sigma):=N_G(\Sigma)\!/\!Z_G(\Sigma)$$
and call it the \textbf{fat Weyl group} of $\Sigma$. The isometric action $(W,\Sigma)$ is called a \textbf{reduction} of $(G,M)$ (induced by $\Sigma$). For minimal sections, $(W,\Sigma)$ is called a \textbf{minimal reduction}.
\end{definition}
\begin{remark}{\ }
\begin{enumerate}
\item One can also define fat sections without requiring them to be embedded. For our purposes however it will be
important that $\Sigma$ is closed in $M$, because then $N_G(\Sigma)$ is a Lie subgroup of $G$. Hence $W(\Sigma)$ is also a Lie group.

\item Every compact Lie group can be realized as a fat Weyl group. This generalizes \cite[Remark 5.6.20]{PT2} and is described at the end of Section \ref{s:resolution}.
\end{enumerate}
\end{remark}

\begin{example}
A minimal reduction of Example \ref{e:standard} is $(\On(k),\R^{k^2})$.
\end{example}
The next proposition is easy to check.
\begin{proposition}\label{p:cohom_copol}
Let $(G,M)$ be an isometric action and let $\Sigma\subseteq M$ be a sufficiently small section. Then $H:=Z_G(\Sigma)$ is a principal isotropy group of $(G,M)$. In particular, all principal isotropy groups along $\Sigma$ coincide. It follows that $W=W(\Sigma)$ acts freely on $\Sigma^\reg$, and if $\Sigma$ is a minimal section, then $\copol(G,M)=\dim(W)$ and
$$\dim\Sigma=\cohom(G,M)+\copol(G,M).$$
\end{proposition}

\section{Properties of Reductions}\label{s:properties}
In this section we generalize several results of \cite[Section 5.2]{GOT}, where orthogonal representations are considered, to arbitrary isometric actions. Interestingly, we obtain the results in a reversed order than in loc.~cit. We start with a metric observation concerning orbit spaces, which is a stronger result than \cite[Theorem 5.9]{GOT}.
In the following let $(G,M)$ be an isometric group action and let $\Sigma$ be a fat section with fat Weyl group $W=W(\Sigma)$.
\begin{theorem}\label{t:orbitspaceisometry}
The orbit spaces $W\backslash\Sigma$ and $G\backslash M$, both endowed with their respective orbital distance metric, are canonically isometric via the map
$$\tilde\iota :W\backslash\Sigma\to G\backslash M,\ W\cdot q\mapsto G\cdot q.$$
\end{theorem}
\begin{proof}
First of all, $\tilde\iota$ is a well defined map: If $q,q'\in\Sigma$ are such that $W\cdot q=W\cdot q'$, then there exists $n\in N_G(\Sigma)\subseteq G$ such that $n\cdot q=q'$. Hence $G\cdot q=G\cdot q'$. Since $\Sigma$ intersects all $G$ orbits, it is clear that $\tilde\iota$ is surjective. Furthermore, the next diagram commutes:
$$\xymatrix{
\Sigma\ar@{->>}[r]^{\iota}\ar@{->>}[d]_{\pi_W} & M\ar@{->>}[d]^{\pi_G}\\
W\backslash\Sigma\ar@{->>}[r]_{\tilde\iota} & G\backslash M.
}$$
As $\Sigma$ is embedded into $M$, the inclusion $\iota$ is continuous. The diagram then implies that $\tilde\iota$ is continuous, too. The distance between two points $G\cdot q$ and $G\cdot q'$ in $G\backslash M$ is the length of a minimal geodesic segment $\gamma$ in $M$ connecting the orbits $G\cdot q$ and $G\cdot q'$. Each such segment is perpendicular to both orbits. If $q,q'$ are both $G$-regular and $q\in\Sigma$, then by properties \ref{pr:1} and \ref{pr:3} of a fat section, $\gamma$ is a segment in $\Sigma$. We may thus further assume $q'\in\Sigma$, and thus $\gamma$ minimizes the distance between $W\cdot q$ and $W\cdot q'$. It follows that $\tilde\iota$ restricted to the open and dense subset $\Sigma^\reg$ (see Proposition \ref{p:copolarproperties} (viii)) is an isometry. By continuity and using that $W\backslash\Sigma$ and $G\backslash M$ are complete metric spaces, we see that $\tilde\iota$ is a surjective isometry.
\end{proof}

\begin{corollary}\label{c:c_isom}
The map $\iota^*:\kc^0(M)^G\to\kc^0(\Sigma)^W,\, f\mapsto f|_\Sigma$ is an isomorphism of Banach algebras, where both spaces are equipped with the corresponding $\|.\|_\infty$-norm.
\end{corollary}
\begin{proof}
Consider the following commuting diagram of Banach algebras associated with the diagram from Theorem \ref{t:orbitspaceisometry}:
$$\xymatrix{
\kc^0(G\backslash M)\ar@{->}[r]^{\tilde\iota^*}\ar@{->}[d]_{\pi_G^*} & \kc^0(W\backslash\Sigma)\ar@{->}[d]^{\pi_W^*}\\
\kc^0(M)^G\ar@{->}[r]_{\iota^*} & \kc^0(\Sigma)^W.
}$$
The top arrow is an isomorphism of Banach algebras, because $G\backslash M\approx W\backslash \Sigma$ and since the assignment $\tilde\iota^*(f)=f\circ\tilde\iota$ is clearly norm preserving. 
The vertical maps are Banach algebra isomorphisms by definition of the orbit space. Hence the bottom arrow is also an isomorphism of Banach algebras.
\end{proof}

\begin{corollary}\label{c:orbitparam}
The fat Weyl group $W$ parameterizes intersections of $G$-orbits with $\Sigma$: For all $q\in\Sigma$ we have $W\cdot q=(G\cdot q)\cap\Sigma$. In particular, $(G\cdot q)\cap\Sigma$ is an extrinsic homogenous submanifold of the spaces $G\cdot q$, $\Sigma$ and $M$ for every $q\in\Sigma$.
\end{corollary}

\begin{corollary}\label{c:isotropy_transitive}
For every $q\in M$ the isotropy group $G_q$ is transitive on the set of all $G$-translates of $\Sigma$ containing $q$. In particular, $G_q$ is transitive on the set of minimal sections through $q$.
\end{corollary}
\begin{proof}
Since $\Sigma$ intersects every orbit we may assume $q\in\Sigma$. Let $g\in G$ be such that $q\in g\cdot\Sigma$. We have to show that there is some $\tilde g\in G_q$ such that $\tilde g\cdot\Sigma=g\cdot\Sigma$ holds.
Since we have $q\in g\cdot\Sigma$, it follows that $g^{-1}\cdot q\in\Sigma$. By Corollary \ref{c:orbitparam} there is some $n\in N_G(\Sigma)$ such that $g^{-1}\cdot q=n\cdot q$ and it follows that $\tilde g:=gn\in G_q$ and $\tilde g\cdot\Sigma=g\Sigma$.
\end{proof}


By property \ref{pr:3} of a fat section $T_p(G\cdot p)$ decomposes orthogonally for every $G$-regular $p\in\Sigma$ into $(T_p(G\cdot p)\cap T_p\Sigma)\oplus\nu_p\Sigma$. More generally:
\begin{proposition}\label{p:orthogonal_decomp}
In all points $q$ of a fat section $\Sigma$ the tangent space $T_qM$ decomposes compatibly and orthogonally in two ways:
$$T_qM=T_q\Sigma\oplus\nu_q\Sigma=T_q(G\cdot q)\oplus\nu_q(G\cdot q).$$
This means that the following decompositions are orthogonal:
\begin{eqnarray*}
T_q\Sigma&=&(T_q\Sigma\cap T_q(G\cdot q))\oplus (T_q\Sigma\cap \nu_q(G\cdot q)),\\ 
\nu_q\Sigma&=&(\nu_q\Sigma\cap T_q(G\cdot q))\oplus (\nu_q\Sigma\cap \nu_q(G\cdot q)).
\end{eqnarray*}
\end{proposition}
The proof is basically the same as of \cite[Lemma 5.10]{GOT}.
\begin{definition}\label{d:kd_ke}
For a given fat section $\Sigma$ and for every $q\in\Sigma$ we define
\begin{eqnarray*}
\kd_q &:=& T_q(G\cdot q)\cap T_q\Sigma \text{ and}\\
\ke_q &:=& T_q(G\cdot q)\cap\nu_q\Sigma.
\end{eqnarray*}
Following \cite{GOT} we extend $\kd$ and $\ke$ to $G$-invariant distributions on $M^\reg$ using property \ref{pr:4} of a fat section. This yields $T_p(G\cdot p)=\kd_p\oplus\ke_p$ for all $p\in M^\reg$.
\end{definition}
\begin{remark}
Due to Proposition \ref{p:orthogonal_decomp}, $T_q(G\cdot q)=\kd_q\oplus\ke_q$ is an orthogonal decomposition for all $q\in\Sigma$ and both $\kd$ and $\ke$ are $W$-invariant (singular) distributions along $\Sigma$.
\end{remark}
\begin{theorem}\label{t:totally_geod_orbits} Let $\Sigma$ be a fat section of $(G,M)$. For every $q\in\Sigma$, the submanifold $W\cdot q\subseteq G\cdot q$ is totally geodesic in $G\cdot q$. Furthermore, for every $\eta\in\nu_q(G\cdot q)\cap T_q\Sigma$ the shape operator $A_\eta$ of $G\cdot q$ leaves the decomposition $T_q(G\cdot q)=\kd_q\oplus\ke_q$ invariant.
\end{theorem}
\begin{proof}
$W\cdot q$ is a submanifold of $M$, and by Proposition \ref{p:orthogonal_decomp} we have
$$T_x(G\cdot q)=\kd_x\oplus\ke_x=(T_x(G\cdot x)\cap T_x\Sigma)\oplus (T_x(G\cdot x)\cap \nu_x\Sigma)$$
for all $x\in W\cdot q$. Therefore $T_x(G\cdot q)$ is invariant under the orthogonal reflection on $T_x\Sigma$. Now the claim follows from the next Lemma, which is \cite[Exercise 8.6.3]{BCO}.
\end{proof}
\begin{lemma}
Let $\Sigma, N$ and $\Sigma\cap N$ be submanifolds of the Riemannian manifold $M$ and suppose that $\Sigma$ is totally geodesic. Suppose that $T_pN$ is invariant under the orthogonal reflection at $T_p\Sigma$ for all $p\in \Sigma\cap N$, then $\Sigma\cap N$ is totally geodesic as a submanifold of $N$ and $A_\eta$, the shape operator of $N$, leaves $T_p(\Sigma\cap N)$ invariant for all $p\in\Sigma\cap N$ and $\eta\in\nu_pN\cap T_p\Sigma$.
\end{lemma}

\begin{remark}
For a polar action the Weyl-group orbits are discrete sets of points and thus trivially totally geodesic. However, if the copolarity is positive and non-trivial, then the orbits of the fat Weyl group are proper positive-dimensional totally geodesic submanifolds in their ambient orbit. So one should expect that the theorem imposes certain restrictions on actions having non-trivial positive copolarity.
\end{remark}

\section{Copolarity and Reductions of the Slice Representation}\label{s:slicerep}
We now generalize \cite[Theorem 5.6]{GOT} from representations to arbitrary isometric group actions, without making any further assumptions. Therefore, our proof follows a rather different approach than the one in loc. cit.
\begin{lemma}\label{l:jacobi_decomp}
Let $\Sigma$ be a totally geodesic submanifold of the Riemannian manifold $M$ and let $\gamma\subseteq\Sigma$ be a geodesic. Then every Jacobi field $J$ along $\gamma$ splits uniquely into Jacobi fields $Y$ and $Z$ along $\gamma$ such that $Y$ is a Jacobi field in $\Sigma$ and $Z$ is perpendicular to $\Sigma$. Furthermore, every derivative of $Z$ is perpendicular to $\Sigma$.
\end{lemma}
\begin{proof}
Consider the orthogonal decomposition
$$J(t)=\underbrace{Y(t)}_{\in T_{\gamma(t)}\Sigma}+\underbrace{Z(t)}_{\in\nu_{\gamma(t)}\Sigma}$$
of $J$. Then $Y$ and $Z$ defined in this way are smooth vector fields along $\gamma$. Since $J$ satisfies the Jacobi equation we have:
\[0=J''+R(J,\dot\gamma,\dot\gamma)=Y''+R(Y,\dot\gamma,\dot\gamma)+Z''+R(Z,\dot\gamma,\dot\gamma).\tag{$\Delta$}\]
Clearly, $Y''$ is tangential to $\Sigma$. Since $\Sigma$ is totally geodesic, $R(Y,\dot\gamma,\dot\gamma)$ is also tangential to $\Sigma$. Since parallel transports of vectors normal to a totally geodesic submanifold stay perpendicular to the submanifold, it follows from the characterization of the covariant derivative by parallel transport that $Z''$ is perpendicular to $\Sigma$. Finally, the expression $R(Z,\dot\gamma,\dot\gamma)$ is perpendicular to $\Sigma$, because for all $v\in T\Sigma$ we have, using the symmetry properties of the curvature tensor,
$$\langle R(Z,\dot\gamma,\dot\gamma),v\rangle=\langle \underbrace{R(v,\dot\gamma,\dot\gamma)}_{\in T\Sigma},Z\rangle=0.$$
By $(\Delta)$, both $Y$ and $Z$ are Jacobi fields and $Y$ is even a Jacobi field of $\Sigma$.
\end{proof}

\begin{theorem}[Slice Theorem]\label{t:slicerep}
If $(G,M)$ is isometric, then for all $q\in M$:
$$\copol(G_q,\nu_q(G\cdot q))\le\copol(G,M).$$
More generally, if $\Sigma$ is a fat section of $(G,M)$ and $q\in\Sigma$, then $V_q:=\nu_q(G\cdot q)\cap T_q\Sigma$ is a fat section of $(G_q,\nu_q(G\cdot q))$. If $W$ is the fat Weyl group of $\Sigma$, then $W_q$ projects canonically onto the fat Weyl group of $V_q$.
\end{theorem}
\begin{proof}
Let $\Sigma$ be a fat section through $q$. Since $V_q$ is a linear subspace of $\nu_q(G\cdot q)$, property \ref{pr:1} of a fat section is already satisfied. Property \ref{pr:2} follows from \ref{pr:3}: There exist $G_q$-regular points in $V_q$ by Lemma \ref{l:sliceregular}. By property \ref{pr:3} and Lemma \ref{l:slice_intersect} it follows that $V_q$ intersects every $G_q$-orbit.\newline We also have property \ref{pr:4}: If $v\in V_q$ is $G_q$-regular, then, after scaling if necessary, we may assume that $p:=\exp_q(v)$ lies in a slice $S_q$ through $q$. Let $g\in G_q$ satisfy $g\cdot v\in V_q$. Then $g\cdot p\in\Sigma$. The $G_q$ regular points in $S_q$ are $G$-regular if viewed as points of $M$. Hence, $p$ is also $G$-regular and therefore $g\cdot\Sigma=\Sigma$. It follows that
$$g\cdot V_q=T_q(g\cdot(\Sigma\cap S_q))=T_q(\Sigma\cap S_q)=V_q.$$
It remains to show property \ref{pr:3} of a fat section. Equivalent to \ref{pr:3} is
$$V_q^\bot\subseteq T_v(G_q\cdot v)$$
for all $G_q$-regular $v\in V_q$. Here $V_q^\bot$ denotes the orthogonal complement of $V_q$ in $\nu_q(G\cdot q)$. As in the proof of property \ref{pr:4} assume that $p=\exp_q(v)$ lies in a slice $S_q$ through $q$. Since $p$ is a $G$-regular and in $\Sigma$, 
property \ref{pr:3} of $\Sigma$ implies $\nu_p\Sigma\subseteq T_p(G\cdot p)$.
Let $w\in V_q^\bot$ be arbitrary. Then $d\exp_q(v)(w)\in\nu_p\Sigma$.
In fact, $d\exp_q(v)(w)=J(1)$ for the Jacobi field $J$ along $\gamma_v(t)=\exp_q(t\cdot v)$ and initial values $J(0)=0$ and $J'(0)=w\in\nu_q\Sigma$. (see \cite[Chapter IX, Theorem 3.1]{L}). By Lemma \ref{l:jacobi_decomp}, $J$ is perpendicular to $\Sigma$. In particular, $$J(1)\in\nu_p\Sigma\subseteq T_p(G\cdot p).$$
Now let $X$ be a $G$-Killing field with $d\exp_q(v)(w)=X_p$. Since $d\exp_q(v)(w)\in T_p S_q$
and $(G\cdot p)\cap S_q=G_q\cdot p$ we may further assume that $X$ is a $G_q$-Killing field. Therefore
$$d\exp_q(v)(w)=X_p\in T_p(G_q\cdot p).$$
Since 
$G_q\cdot p=\exp_q(G_q\cdot v)$, we get
$$T_p(G_q\cdot p)=d\exp_q(v)(T_v(G_q\cdot v)).$$
It follows that $w\in T_v(G_q\cdot v)$, because $d\exp_q(v)$ is bijective.

We have therefore proved that $V_q$ is a fat section of $(G_q,\nu_q(G\cdot q))$. The actions $(G,M)$ and $(G_q,\nu_q(G\cdot q))$ have the same cohomogeneity, and $\dim V_q\le\dim\Sigma$. Choosing $\Sigma$ as a minimal section, it therefore follows that the copolarity of the slice representation is less than or equal to the copolarity of the $G$-action on $M$.

The fat Weyl group of $V_q$ is given by
$$W(V_q)=N_{G_q}(V_q)\!/\!Z_{G_q}(V_q).$$
We first show $N_{G_q}(V_q)=N_{G_q}(\Sigma)=(N_G(\Sigma))_q.$ Let $g\in N_G(\Sigma)\cap G_q=N_{G_q}(\Sigma)$ be arbitrary. Then $g$ leaves both $T_q\Sigma$ and $\nu_q(G\cdot q)$ invariant. Therefore, $T_q\Sigma\cap\nu_q(G\cdot q)=V_q$ is also left invariant and it follows that $g\in N_{G_q}(V_q)$. Conversely, for $g\in N_{G_q}(V_q)$, again as in the proof of property \ref{pr:4}, it follows that $g\cdot\Sigma=\Sigma$ and hence $g\in N_{G_q}(\Sigma)$.
Now it is easy to see that
$$Z_G(\Sigma)=Z_{G_q}(\Sigma)\subseteq Z_{G_q}(V_q).$$
The commuting diagram below implies that $W(\Sigma)_q$ projects canonically onto $W(V_q)$:
$$\xymatrix{
N_{G_q}(\Sigma)\ar@{=}[r]\ar@{->>}_\pr[d] & N_{G_q}(V_q)\ar@{->>}^\pr[d]\\
W(\Sigma)_q\ar@{.>>}[r] & W(V_q).
}$$
\end{proof}
\begin{remark}
For a minimal section $\Sigma$, we do not know wether $V_q$ is necessarily a minimal section of the slice representation $(G_q,\nu_q(G\cdot q))$, or not. However, the above proof shows:
\end{remark}
\begin{corollary}\label{c:slicerep}
If $\Sigma$ is a pre-section of $(G,M)$, then $V_q=\nu_q(G\cdot q)\cap T_q\Sigma$ is a pre-section of $(G_q,\nu_q(G\cdot q))$. If $\Sigma$ is a sufficiently small section, then $V_q$ is also sufficiently small and $W(V_q)=W_q$.
\end{corollary}

\section{Stability of Copolarity under Reductions}\label{s:stability}
We next show that the copolarity of a reduction $(W,\Sigma)$ is equal to that of $(G,M)$. We start with a Lemma, which may be interesting in its own right.
\begin{lemma}\label{l:W-reg_is_G-reg}
If $\Sigma$ is a fat section of an isometric action $(G,M)$, then the $G$-regular points in $\Sigma$ are $W(\Sigma)$-regular and viceversa.
\end{lemma}
\begin{proof}
According to Proposition \ref{p:copolarproperties} (viii) the set of $G$-regular points is open and dense in $\Sigma$. If we can show that the $G$-regular points in $\Sigma$ all have the same $W(\Sigma)$-orbit type, then they must be $W(\Sigma)$-regular. This is because the $W(\Sigma)$-regular points are open and dense in $\Sigma$, too.
Let $p\in\Sigma$ be an arbitrary $G$-regular point. Property \ref{pr:4} of a fat section implies $Z_G(\Sigma)\subseteq G_p\subseteq N_G(\Sigma)$, and thus $(N_G(\Sigma))_p=G_p$. Let $q$ be another $G$-regular point in $\Sigma$. Connect $q$ with $G\cdot p$ by a $G$-transversal geodesic $\gamma$. Then by properties \ref{pr:1} and \ref{pr:3} of a fat section, $\gamma$ is a geodesic of $\Sigma$. We may assume that $\gamma(0)=q$ and $\gamma(1)=g\cdot p$ for some $g\in G$. By property \ref{pr:4} again,we have $g\in N_G(\Sigma)$. Since $G_q=G_{g\cdot p}=gG_p g^{-1}$ we have that both $p$ and $q$ are of the same $W(\Sigma)$-orbit-type.

Conversely, let $q\in\Sigma$ be an arbitrary $W(\Sigma)$-regular point. By Theorem \ref{t:slicerep}, $V_q$ is a fat section of $(G_q,\nu_q(G\cdot q))$ and $W_q$ projects canonically onto the fat Weyl group $W(V_q)$ of $V_q$. Proposition \ref{p:orthogonal_decomp} shows that $V_q$ is also the representation space for the slice representation of $(W(\Sigma),\Sigma)$ in $q$. By assumption, $W_q$ acts trivially on $V_q$. Since $W(V_q)$ acts effectively on $V_q$ by definition, the group $W(V_q)$ must be trivial. In particular, $(G_q,\nu_q(G\cdot q))$ is a polar representation with generalized Weyl group $W(V_q)$. According to \cite[Corollary 5.6.22]{PT2} the latter is a Weyl group in the classical sense. However, a polar representation with trivial Weyl group must be trivial itself. Thus $G_q$ acts trivially on $\nu_q(G\cdot q)$, and in conclusion $q$ is $G$-regular.
\end{proof}

\begin{theorem}[Stability theorem]\label{t:stability}
Let $(G,M)$ be an isometric action and let $\Sigma$ be an arbitrary fat section. Then a subset $\Sigma'\subseteq\Sigma$ is a fat section of $(G,M)$ if and only if it is a fat section of $(W(\Sigma),\Sigma)$. It follows that
$$\copol(G,M)=\copol(W(\Sigma),\Sigma).$$
If $\Sigma$ is a minimal section, then the copolarity of $(W(\Sigma),\Sigma)$ is trivial.
\end{theorem}

\begin{proof}
First of all, if $\Sigma'$ is complete and connected, totally geodesic and embedded in $\Sigma$, then it also has these properties as a submanifold of $M$ and viceversa. If $\Sigma'$ intersects every $G$-orbit, then it also intersects every $W(\Sigma)$-orbit, because the latter are the intersections of $G$-orbits with $\Sigma$ and we have $\Sigma'\subseteq\Sigma$ (Corollary \ref{c:orbitparam}). Conversely, if $\Sigma'$ intersects every $W(\Sigma)$-orbit, then it also intersects every $G$-orbit, because every $G$-orbit contains a $W(\Sigma)$-orbit. Next, by Lemma \ref{l:W-reg_is_G-reg}, we need not distinguish between $G$-regular and $W(\Sigma)$-regular points in $\Sigma'$. We have for every regular $p\in\Sigma$:
$$\nu_p(G\cdot p)=\nu^\Sigma_p(W(\Sigma)\cdot p).$$
Therefore, $\nu_p(G\cdot p)\subseteq T_p\Sigma'$ is equivalent to $\nu^\Sigma_p(W\cdot p)\subseteq T_p\Sigma'$, for every regular $p\in\Sigma'$.
Finally, let $p\in\Sigma'$ be regular and let $g\in G$ be such that $g\cdot p\in\Sigma'$. Since $\Sigma'\subseteq\Sigma$, it follows that $g\in N_G(\Sigma)$. Now it is clear that $\Sigma'$ has property \ref{pr:4} of a fat section with respect to $(G,M)$ if and only if it it has this property with respect to $(W(\Sigma),\Sigma)$.
\end{proof}

\section{A Remark on Variational Completeness and Co-Completeness}\label{s:varcomp}
\noindent A main result of this section is that variational completeness of an isometric action is inherited to every reduction of that action, and conversely variational completeness of a reduction extends to the variational completeness of the original action.
As a slight excursion we also generalize \cite[Theorem 4.1]{GOT} in such a way that we relax the condition that the fat section $\Sigma$ has to be flat to the condition that $\Sigma$ has no conjugate points. This applies to more general situations, like $\sec(\Sigma)\le0$.
\begin{definition}
Let $N$ be a submanifold of $M$. An \textbf{$N$-geodesic} $\gamma:[0,\varepsilon)\to M$ is a geodesic of $M$ which emanates perpendicularly from $N$. An \textbf{$N$-Jacobi field} $J$ is a Jacobi field (along an $N$-geodesic $\gamma$) which is induced by a variation of $N$-geodesics.
\end{definition}
One can show that if $\gamma(0)=p\in N$ and $v=\gamma'(0)$, then $J$ is an $N$-Jacobi field if and only if it is a Jacobi field satisfying $J(0)\in T_pN$ and $J'(0)+A_vJ(0)\in\nu_pN$. Here $A_v$ denotes the shape operator of $N$ in the direction of $v$. Furthermore, the vector space $\kj^N(\gamma)$ of all $N$-Jacobi fields along $\gamma$ is isomorphic to $T_pM=T_pN\oplus\nu_pN$ via $J\mapsto J(0)+(J'(0)+A_vJ(0))$.

We fix a fat section $\Sigma$ of $(G,M)$ and let $N:=G\cdot p$ denote a fixed principal orbit with $p\in\Sigma$. For $v\in\nu_p N$ let $\gamma_v(t):=\exp_p(tv)$. The following lemmas as well as their proofs are \cite[Lemma 4.3 and Lemma 4.4]{GOT}. The second one characterizes under which conditions an $N$-Jacobi field is perpendicular to a given fat section, whereas the first one shows that every $N$-Jacobi field, induced by a $G$-Killing field and with the proper initial values, always satisfies this condition. Note that $\sec(\Sigma)$ may be arbitrary.

\begin{lemma}[{\cite[Lemma 4.3]{GOT}}]\label{l:Killing_Jacobi}
Let $J$ be an $N$-Jacobi field along $\gamma_v$ with $J(0)\in\ke_p$. If $J$ is the restriction of a $G$-Killing field on $M$ to $\gamma_v$, then $J$ satisfies $J'(0)+A_v J(0)=0$.
\end{lemma}
\begin{lemma}[{\cite[Lemma 4.4]{GOT}}]\label{l:Jacobi_orth}
Let $J$ be an $N$-Jacobi field along $\gamma_v$ such that $J(0)\in\ke_p$. Then $J$ is always orthogonal to $\Sigma$ if and only if $J'(0)+A_v J(0)=0$.
\end{lemma}
%
With these lemmas we get a refined decomposition of $\kj^N(\gamma)$:
\begin{proposition}\label{p:Jacobi}
Let $\tilde N:=W(\Sigma)\cdot p$ and denote the $\tilde N$-Jacobi fields in $\Sigma$ by $\kj^{\tilde N}(\gamma)$. Then
$$\kj^N(\gamma)=\kj_0^N(\gamma)\oplus\kj_{\kd}^N(\gamma)\oplus\kj_{\ke}^N(\gamma), \text{ where}$$
$$\begin{array}{lclll}
\kj_0^N(\gamma)&:=&\{J\in\kj^N(\gamma)\mid J(0)=0,J'(0)\in\nu_pN\}&=&\kj_0^{\tilde N}(\gamma),\\
\kj_{\kd}^N(\gamma)&:=&\{J\in\kj^N(\gamma)\mid J(0)\in\kd_p,\, J'(0)=-A_v J(0)\}&=&\kj_{\kd}^{\tilde N}(\gamma),\\
\kj_{\ke}^N(\gamma)&:=&\{J\in\kj^N(\gamma)\mid J(0)\in\ke_p,\, J'(0)=-A_v J(0)\}& & \\
&=&\{X|_\gamma\mid X \text{ is a } G \text{-Killing field and } X_p\in\ke_p\}. & &
\end{array}$${\ }\smallskip\\
In particular, if $J=J_0+J_\kd+J_\ke$ is an $N$-Jacobi field represented with respect to the above decomposition, then, in view of Lemma \ref{l:jacobi_decomp}, $J_0+J_\kd$ is the part of $J$ which is everywhere tangential to $\Sigma$ and $J_\ke$ is part of $J$ which is everywhere perpendicular to $\Sigma$. 
\end{proposition}
\begin{proof}
The decomposition follows from the isomorphism $\kj^N(\gamma)\simeq T_pN\oplus\nu_pN$ and because of $T_p(G\cdot p)=\kd_p\oplus\ke_p$ (see Definition \ref{d:kd_ke}). Note that Theorem \ref{t:totally_geod_orbits} implies that $A_v$ leaves $\kd_p$ invariant. This shows that every element $J_\kd$ of $\kj_\kd^N(\gamma)$ is everywhere tangential to $\Sigma$, because $J_\kd(0)\in\kd_p$ and $J_\kd'(0)=-A_vJ_\kd(0)\in\kd_p$. It is also clear that every element $J_0$ of $\kj_0^N(\gamma)$ is tangential to $\Sigma$, because of $J_0(0)=0$ and $J_0'(0)\in\nu_pN$. We next show that $J_0$ and $J_\kd$ are $\tilde N$-Jacobi fields. First of all, $\gamma$ is a geodesic in $M$ which starts in $\Sigma$ and since $\gamma'(0)\in\nu_pN\subseteq T_p\Sigma$ it is also tangential to $\Sigma$. Since $\Sigma$ is totally geodesic in $M$, it follows that $\gamma$ is a geodesic of $\Sigma$ and furthermore, $\gamma$ is a $\tilde N$-geodesic. Using Lemma \ref{l:jacobi_decomp} we see that $J_0$ and $J_\kd$ are Jacobi fields on $\Sigma$. For $J_0$ we now have to show  $J'(0)\in\nu_p\tilde N$. But this is clear since we have $\nu_pN=\nu_p\tilde N$. Concerning $J_\kd$, we have that $J_\kd(0)\in\kd_p=T_p\tilde N$ and if $\tilde A$ denotes the shape operator of $\tilde N$, then
$$J_\kd'(0)+\tilde A_vJ_\kd(0)=J_\kd'(0)+A_vJ_\kd(0)=0,$$
where we have used that $\tilde A_v=A_v|_{\kd_p}$, because $\tilde
N$ is totally geodesic in $N$, by Theorem
\ref{t:totally_geod_orbits} again. By tracing the previous arguments
backwards, we obtain that in fact the equalities
$\kj^N_0(\gamma)=\kj^{\tilde N}_0(\gamma)$ and
$\kj^N_\kd(\gamma)=\kj^{\tilde N}_\kd(\gamma)$ hold. The statements
concerning $\kj_\ke^N(\gamma)$ are direct consequences of Lemma
\ref{l:Killing_Jacobi} and Lemma \ref{l:Jacobi_orth}.

\end{proof}
\begin{definition}
An isometric action $(G,M)$ is \textbf{variationally complete} if for every $G$-orbit $N$, every $N$-geodesic $\gamma$ and every $N$-Jacobi field along $\gamma$, which vanishes for some $t_0>0$, is the restriction of a $G$-Killing field to $\gamma$.
\end{definition}
It suffices to consider principal orbits only in order to show that an isometric action is variationally complete. This fact seems to be known in the literature. For instance, in \cite{GOT} this is implicitly assumed in the characterization of variational completeness via $\covar(G,M)=0$ (see below). A proof can be found in \cite[Remark 5.5]{LT2}\begin{footnote}{I would like to thank Alexander Lytchak for giving me this reference.}\end{footnote}.

\begin{theorem}\label{t:varcomp}
An isometric action $(G,M)$ is variationally complete if and only if a minimal reduction $(W(\Sigma),\Sigma)$ is variationally complete.
\end{theorem}
\begin{proof}
In the following let $p\in\Sigma$ be a regular point. Due to Lemma \ref{l:W-reg_is_G-reg}, $G$- and $W$-regular points are the same. Put
$$N:=G\cdot p \text{ and } \tilde N:=W(\Sigma)\cdot p$$
and let $\gamma$ be an arbitrary $\tilde N$-geodesic starting in $p$.

Suppose that $(G,M)$ is variationally complete. If $J\in\kj^{\tilde N}(\gamma)$ satisfies $J(t_0)=0$ for some $t_0>0$, then we can view $J$ as an $N$-Jacobi field along the $N$-geodesic $\gamma$, according to Proposition \ref{p:Jacobi}. By variational completeness of $(G,M)$, there is a $G$-Killing field $X$ such that $J=X|_\gamma$. Let now $\pr_\Sigma X$ denote the orthogonal projection of $X$ onto $\Sigma$. By \cite[Theorem 1]{Mag1} this is a $W$-Killing field on $\Sigma$ (here we use that $\Sigma$ is a minimal section). Since $X(\gamma(t))=J(t)\in T_{\gamma(t)}\Sigma$ and therefore $J(t)=\pr_\Sigma X(\gamma(t))$, we may conclude that $J$ is the restriction of a $W$-Killing field to $\gamma$.

For the converse direction, suppose now that $(W(\Sigma),\Sigma)$ is variationally complete. Let $p\in M$ be an arbitrary regular point and $\gamma$ an $N$-geodesic starting in $p$. Without loss of generality, we may assume that $p\in\Sigma$ and that $\gamma$ is an $\tilde N$-geodesic (a suitable translate $g\cdot\Sigma$ contains $p$ and hence $\gamma$, and the minimal reduction $(W(g\cdot\Sigma),g\cdot\Sigma)$ is also variationally complete).
We decompose an arbitrary $N$-Jacobi field $J$, which vanishes for some $t_0>0$, according to Proposition \ref{p:Jacobi} into the three parts $J=J_0+J_\kd+J_\ke$. The proposition tells us that $J_\ke$ is already induced by a $G$-Killing field. From
$$0=J(t_0)=\underbrace{J_0(t_0)+J_\kd(t_0)}_{\in T_p\Sigma}+\underbrace{J_\ke(t_0)}_{\in \nu_p\Sigma}$$
and the variational completeness of $(W(\Sigma),\Sigma)$ it follows that $J_0+J_\kd$ is induced by an $N(\Sigma)$-Killing field. But such a field is also a $G$-Killing field and it follows that $J$ is the restriction of a $G$-Killing field to $\gamma$.
\end{proof}

\begin{corollary}\label{c:varcomp}
An isometric action $(G,M)$ is variationally complete if and only if some (and hence any) reduction $(W(\Sigma),\Sigma)$ is variationally complete.
\end{corollary}
\begin{proof}
According to Theorem \ref{t:stability}, $(G,M)$ and $(W(\Sigma),\Sigma)$ have a common minimal reduction $(W(\Sigma'),\Sigma')$ with $\Sigma'\subseteq\Sigma$. Hence, we may apply Theorem \ref{t:varcomp} to $(G,M)$ and $(W(\Sigma'),\Sigma')$ and then to $(W(\Sigma),\Sigma)$ and $(W(\Sigma'),\Sigma')$ and vice versa.
\end{proof}
\begin{remark}
Theorem \ref{t:varcomp} and Corollary \ref{c:varcomp} can also be deduced from \cite[Theorem 1.3]{LT2} and our Theorem \ref{t:orbitspaceisometry}. In fact, the first result states that variational completeness only depends on the metric properties of $G\backslash M$, which by the second result is isometric to $W\backslash\Sigma$ for any reduction $(W,\Sigma)$ of $(G,M)$.
\end{remark}
\begin{corollary}
If $(G,M)$ is a polar and variationally complete action, then
every section is free of conjugate points. In particular, if $M$
is a Riemannian manifold of non-negative Ricci curvature or compact
and of non-negative scalar curvature, then a variationally complete
action on $M$ is polar, if and only if it is hyperpolar.
\end{corollary}

\begin{proof}
Polarity implies that the generalized Weyl group $W(\Sigma)$ of
any section $\Sigma$ is discrete. Furthermore, a Lie group acts
variationally complete if and only if its identity component does. However,
if the trivial group acts variationally complete, this only means
that every Jacobi field which vanishes in two different points,
vanishes entirely. Hence, there are no conjugate points in
$\Sigma$. By being totally geodesic, $\Sigma$ inherits the curvature conditions of $M$. By a result of Mendonca and Zhou,
\cite[Corollary 1]{MZ}, resp. Green \cite{G}, we deduce from the
above that $\Sigma$ has to be flat.
\end{proof}

\begin{remark}
Conlon proved in \cite{C} that hyperpolar actions are
variationally complete. In general, the converse is false. Take, for instance, the action of the trivial group on a non-flat space of non-positive curvature. This action is variationally complete and polar, but not hyperpolar. However, Lytchak and Thorbergsson proved in \cite{LT}, that variationally
complete actions on manifolds of non-negative curvature are
hyperpolar.
\end{remark}

We briefly recall the notion of variational co-completeness, which has been introduced in \cite{GOT}. Let $N=G\cdot p$ denote an arbitrary principal orbit and consider the isomorphism $\kj^N(\gamma)\simeq T_pN\oplus\nu_pN$. For a subspace $U_p\subseteq T_pM$ consider the condition:\medskip\\
\begin{tabular}{ll}
(P) & for every $N$-geodesic $\gamma$ and every $J\in\kj^N(\gamma)$, vanishing in some $t_0>0$ with\\
 & $(J(0),J'(0)+A_v J(0))\,\bot\, U_p$ it follows that $J=X|_\gamma$ for some $G$-Killing field $X$.
\end{tabular}
\smallskip\\
If $U_p$ satisfies condition $(P)$, then $g_*U_p$ satisfies this condition in $g\cdot p$. Furthermore, $U_p=T_pM$ always satisfies condition $(P)$.
\begin{definition}
We write $\covar_N(G,M)\le\dim U_p$, if $U_p$ satisfies condition
$(P)$. We say that the \textbf{variational co-completeness} of
$(G,M)$ is less than or equal to $k$, if $\covar_N(G,M)\le k$ holds
for all principal orbits $N$. We also write $\covar(G,M)\le k$.
\end{definition}
A canonical choice for $U_p$ is always
$T_p\Sigma=\nu_pN\oplus\kd_p$, where $\Sigma$ denotes a fat section
through $p$. This is due to Proposition \ref{p:Jacobi}. In
particular, we always have
$$\covar(G,M)\le\cohom(G,M)+\copol(G,M).$$
This estimate can sometimes be considerably improved as in the following result, which is a generalization of \cite[Theorem 4.1]{GOT}. We note however that one only has to replace the condition $\sec(\Sigma)=0$ in the proof of \cite[Lemma 4.2]{GOT} by the condition that $\Sigma$ has no conjugate points. This occurs, for instance, whenever $\sec(M)\le0$.
\begin{theorem}
Let $(G,M)$ be an isometric action and $\Sigma\subseteq M$ a $k$-section. If $\Sigma$ is free of conjugate points in the induced metric, then $\covar(G,M)\le k$. In particular,
$$\covar(G,M)\le\copol(G,M).$$
\end{theorem}
 We even obtain Corollary 4.5 of loc. cit. under these relaxed conditions:
\begin{corollary}
Let $(G,M)$ be an isometric action and let $\Sigma$ be a pre-section, wich we assume to have no conjugate points. Let $N$ be a principal orbit and let $p\in N\cap\Sigma$. Then $\kd_p=T_p N\cap T_p\Sigma$ has property $(P)$.
\end{corollary}

\section[Global Resolutions]{Global Resolutions of Isometric Actions with Respect to Fat Sections}\label{s:resolution}
In this section we define the (global) resolution $M_\Sigma$ of an isometric action $(G,M)$ with respect to an arbitrary fat section $\Sigma$. This is related to the \emph{core resolution construction} of Grove and Searle in \cite{GS}. The reason, why $M_\Sigma$ is called a resolution, is that it is a $G$-space whose isotropy groups are smaller than those of $(G,M)$. Roughly speaking, the $G$-orbits on $M_\Sigma$ are less singular than the $G$-orbits on $M$. In the following, let $(G,M)$ be an isometric action and let $\Sigma$ be a fat section. Put $N=N_G(\Sigma)$ and $H=Z_G(\Sigma)$. Then $W=N\!/\!H$ is the fat Weyl group of $\Sigma$. Since $\Sigma$ is a $W$-space, we may form the associated bundle $G\!/\!H\times_W\Sigma\twoheadrightarrow G\!/\!N$ with fibre $\Sigma$, where $G\!/\!H\times_W\Sigma$ is the orbit space under the diagonal $W$-action on $G\!/\!H\times \Sigma$ given by $nH\cdot (gH,s):=(gn^{-1}H,n\cdot s)$.
Its total space is a $G$-space with respect to the $G$-action $l\cdot[gH,s]:=[lgH,s]$.

\begin{definition}
The \textbf{resolution} of $(G,M)$ with respect to $\Sigma$ is defined as
$$M_\Sigma:=G\!/\!H\times_W\Sigma.$$
If $\Sigma$ is a minimal section, we call $M_\Sigma$ a \textbf{minimal resolution}.
\end{definition}

We now list some features related to $M_\Sigma$ (c.f. \cite{GS}, Theorem 2.1):
\begin{theorem}\label{t:weylcover}
Let $\varphi:G\times M\to M$ denote the group action $(G,M)$. Then
\begin{enumerate}
\item The group action $\varphi$ induces a smooth and surjective $G$-equivariant map:
$$\tilde\varphi:M_\Sigma\twoheadrightarrow M,[gH,s]\mapsto g\cdot s.$$

\item The isotropy group of a point $[eH,s]\in M_\Sigma=G\!/\!H\times_W\Sigma$ is given by:
$$G_{[eH,s]}=N\cap G_s=N_{G_s}(\Sigma).$$

\item $\Sigma$ is canonically $N$-equivariantly immersed into $M_\Sigma$ via the map $s\mapsto [eH,s]$. The image $\tilde\Sigma$ is embedded into $M_\Sigma$ because it is a fibre of $M_\Sigma\twoheadrightarrow G\backslash N$, and furthermore it intersects every $G$-orbit on $M_\Sigma$. It follows that $\tilde\varphi$ restricts to a $W$-equivariant diffeomorphism between $\tilde\Sigma$ and $\Sigma$.

\item The set of $G$-regular points $(M_\Sigma)^\reg$ can be identified with $G\!/\!H\times_W\Sigma^\reg$, and $\tilde\varphi$ restricts to a $G$-equivariant diffeomorphism from $(M_\Sigma)^\reg$ onto $M^\reg$.
This yields a bundle with structure group $W$ and totally geodesic fibres $g\cdot\Sigma^\reg,\, g\in G$:
$$\pi:M^\reg\twoheadrightarrow G\!/\!N,\, g\cdot s\mapsto gN.$$

\item The orbit spaces $G\backslash M_\Sigma$ and $G\backslash M$ are canonically homeomorphic.

\item $d\tilde\varphi_{[eH,s]}:T_{[eH,s]}M_\Sigma\to T_sM$ is a linear isomorphism if and only if
\begin{equation}\tag{$*$}
T_s(G\cdot s)+T_s\Sigma=T_s M.
\end{equation}
This is furthermore equivalent to $G_s\subseteq N$ and also to $(G_s)^\circ=(N\cap G_s)^\circ$.\linebreak
$\tilde\varphi$ is a $G$-equivariant diffeomorphism if and only if $(*)$ is satisfied for all $s\in\Sigma$.

\item The $G$-translates of $\tilde\Sigma$ foliate $M_\Sigma$.
\end{enumerate}
\end{theorem}
\begin{proof}
(i): If $[gH,s]=[\tilde gH,\tilde s]\in M_\Sigma$, then there is some $n\in N$ and $h\in H$ with
$$(\tilde g,\tilde s)=(gn^{-1}h,n\cdot s).$$
It follows that
$$\tilde g\cdot \tilde s=g\underbrace{n^{-1}hn}_{\in H}\cdot s=g\cdot s$$
and we have shown that $\tilde\varphi$ is well defined. Since $\Sigma$ intersects every orbit, it follows that $\varphi$ restricted to $G\times\Sigma$ maps onto $M$. Furthermore, $H$ acts trivially on $\Sigma$, and thus
$$G\!/\!H\times \Sigma^\reg\to M^\reg,\ (gH,s)\mapsto g\cdot s$$
(again denoted by $\varphi$) is still surjective. The following diagram commutes:
$$\xymatrix{
G\!/\!H\times\Sigma\ar@{->>}[rd]^{\varphi}\ar@{->>}_\pr[d]&\\
M_\Sigma\ar@{->}[r]_{\tilde\varphi}&M.
}$$
From this we can read off that $\tilde\varphi$ is also surjective and $G$-equivariant, and since the vertical map is a surjective submersion, it follows that $\tilde\varphi$ is smooth.

(ii): Let $g\in G_{[eH,s]}$ be arbitrary. Then there exists some $n\in N$ and $h\in H$ such that $(g,s)=(n^{-1}h,n\cdot s)$. This implies $n\in G_s$ and therefore $gh^{-1}\in G_s$. Since $H\subseteq G_s$, it follows that $g\in G_s\cap N$. If conversely $g\in G_s\cap N$, then
$$g\cdot[eH,s]=[gH,s]=[eH,g^{-1}\cdot s]=[eH,s],$$
showing that $g\in G_{[eH,s]}$. 

(iii): This statement is easily verified.

(iv): The first part follows from (ii) and (iii). It remains to show that $\tilde\varphi|_{(M_\Sigma)^\reg}$ is injective with smooth inverse. Suppose that $g\cdot s=\tilde g\cdot\tilde s$ for $g,\tilde g\in G$ and $s,\tilde s\in\Sigma^\reg$. Then $\tilde s=\tilde g^{-1}g\cdot s$ and property \ref{pr:4} of a fat section implies $n:=\tilde g^{-1}g\in N$. Hence,
$$[gH,s]=[gn^{-1}H,n\cdot s]=[\tilde g,\tilde s].$$
By property \ref{pr:3} of a fat section, $\Sigma$ is transversal to every principal orbit. Using (vi) it follows that $\tilde\varphi|_{(M_\Sigma)^\reg}$ is a submersion and thus a diffeomorphism.

(v): The map $f:G\backslash M_\Sigma\to G\backslash M,\ G\cdot[eH,s]\mapsto G\cdot s$ is well defined and makes
$$\xymatrix{
M_\Sigma\ar@{->>}[r]^{\tilde\varphi}\ar@{->>}_\pr[d]& M\ar@{->>}^\pr[d]\\
G\backslash M_\Sigma\ar@{->}_f[r]&G\backslash M
}$$
commute. Hence $f$ is continuous and surjective. It is also easy to see that $f$ is injective. To show that $f^{-1}$ is continuous, we write it as a composition of continuous maps:
$$\begin{array}{ccccccc}
G\backslash M & \stackrel{\tilde\iota^{-1}}{\to} & W\backslash\Sigma & \to & W\backslash\tilde\Sigma & \to & G\backslash M_\Sigma,\\
G\cdot s&\mapsto & W\cdot s &\mapsto &W\cdot[eH,s]&\mapsto & G\cdot[eH,s].
\end{array}
$$
Here $\tilde\iota$ is the map from Theorem \ref{t:orbitspaceisometry} and the other two maps are the continuous injections induced by the continuous maps $\Sigma\hookrightarrow\tilde\Sigma$, resp. $\tilde\Sigma\hookrightarrow M_\Sigma$, both of which appear in (iii).

(vi): From the diagram in the proof of (i) we see that $d\tilde\varphi_{[eH,s]}:T_{[eH,s]}M_\Sigma\to T_sM$ is surjective if and only if
$d\varphi_{(eH,s)}:T_{(eH,s)}G\!/\!H\times\Sigma\to T_s M$ is surjective. We have
$$d\varphi_{(eH,s)}(X+\gothh,v)=X_s+v,$$
where $X_s$ is the value of the Killing field induced by $X\in\gothg$ on $M$ in $s$. This yields
$\im(d\varphi_{(eH,s)})=T_s(G\cdot s)+T_s\Sigma$, and thus $d\varphi_{(eH,s)}$ is onto if and only if $(*)$ holds.

By Proposition \ref{p:orthogonal_decomp}
\begin{equation}\tag{$**$}
T_s(G\cdot s)+ T_s\Sigma=T_s(G\cdot s)\oplus(T_s\Sigma\cap\nu_s(G\cdot s)).
\end{equation}
Since the decomposition on the right is orthogonal, $(*)$ is equivalent to $\nu_s(G\cdot s)\subseteq T_s\Sigma.$
This in turn is equivalent to the statement that $G_s\subseteq N$. In fact, since the $G_s$-regular points in $\nu_s(G\cdot s)$ correspond to $G$-regular points in $M$ under the exponential map, it follows from property \ref{pr:4} of a fat section that, if $\nu_s(G\cdot s)\subseteq T_s\Sigma$ holds, then $G_s\subseteq N$. Conversely, if $G_s\subseteq N$ then, according to the Slice Theorem \ref{t:slicerep}, $\nu^\Sigma_s(W\cdot s)$ is a $G_s$-invariant subspace of $\nu_s(G\cdot s)$. However, this just means $\nu_s(G\cdot s)=\nu^\Sigma_s(W\cdot s)\subseteq T_s\Sigma.$

Again by Proposition \ref{p:orthogonal_decomp}, we have $T_s(G\cdot s)+T_s\Sigma=T_s\Sigma\oplus (T_s(G\cdot s)\cap\nu_s(\Sigma)).$
Thus, $(*)$ is furthermore equivalent to
\begin{equation}\tag{$***$}
\dim M=\dim\Sigma+(\dim(G\cdot s)-\dim(W\cdot s))
\end{equation}
Let $(\cdot)_\princ$ denote a principal isotropy group for the action in parentheses. Then
$$\dim\Sigma=\cohom(G,M)+\dim W-\underbrace{\dim (W,\Sigma)_\princ}_{=\dim (G,M)_\princ-\dim H}$$
The right hand side of $(***)$ is therefore equal to:
\begin{eqnarray*}
&&\cohom(G,M)+\dim W-\dim (W,\Sigma)_\princ+(\dim G-\dim G_s-\dim W+\dim W_s)\\
&=& \dim M+\dim H-\dim G_s+\underbrace{\dim W_s}_{=\dim(G_s\cap N)-\dim H}\\
&=&\dim M+\dim(G_s\cap N)-\dim G_s.
\end{eqnarray*}
It follows that $(***)$ is equivalent to $\dim(G_s\cap N)=\dim G_s$, or $(G_s)^\circ=(G_s\cap N)^\circ$.

Suppose that $\tilde\varphi$ is a local diffeomorphism. Since $\tilde\varphi$ restricted to $(M_\Sigma)^\reg$ is a diffeomorphism onto $M^\reg$ and since the regular points form an open and dense subset of their surrounding space, it follows that $\tilde\varphi$ is a diffeomorphism from $M_\Sigma$ onto $M$.

(vii): Let $q:=[gH,s]\in M_\Sigma$ be arbitrary. Due to Corollary \ref{c:isotropy_transitive}, $G_q=g(N\cap G_s)g^{-1}$ is transitive on the set of $G$-translates of $\tilde\Sigma$ that contain $q$. Clearly, $g\cdot\tilde\Sigma$ contains $q$. For an arbitrary $gng^{-1}\in G_q$, where $n\in N(\Sigma)\cap G_s$, we have
$(gng^{-1})\cdot(g\cdot\tilde\Sigma)=(gn)\cdot\tilde\Sigma=g\cdot\tilde\Sigma.$
Therefore, the only $G$-translate through $q$ is $g\cdot\tilde\Sigma$.
\end{proof}

\begin{corollary}
If $(G,M)$ has only principal or exceptional orbits, then $M_\Sigma\simeq M$.
\end{corollary}
\begin{proof}
(Compare with {\cite[Corollary 2.4]{GS}}). Let $q\in\Sigma$ be arbitrary. According to Lemma \ref{l:sliceregular} there is some $G$-regular point $p\in\Sigma$ in a slice around $q$. We thus have $G_p\subseteq G_q$, and by assumption $(G_q)^\circ=(G_p)^\circ$. Since $p\in\Sigma$ is $G$-regular, property \ref{pr:4} of a fat section implies $G_p\subseteq N$. This yields:
$(G_p)^\circ\subseteq(N(\Sigma)\cap G_q)^\circ\subseteq(G_q)^\circ=(G_p)^\circ$, and the claim follows from Theorem \ref{t:weylcover} (vi).
\end{proof}

So far we have considered $M_\Sigma$ only as a smooth manifold without any Riemannian metric on it. It is natural to demand that $G$ should act isometrically on $M_\Sigma$. Furthermore, the Riemannian metric on $M_\Sigma$ should be induced by a product metric on $G\!/\!H\times\Sigma$. Hence, we consider $(G$-$W)$-invariant metrics on $G\!/\!H$ (cf. Section \ref{s:invariant}).
\begin{proposition}\label{p:resolution}
Suppose that $G\!/\!H$ carries a $(G$-$W)$-invariant Riemannian metric and $\Sigma$ the Riemannian metric induced by $M$. Then $M_\Sigma$, endowed with the Riemannian metric submersed from $G\!/\!H\times\Sigma$, has the following properties:
\begin{enumerate}
\item $(G,M_\Sigma)$ is an isometric action.
\item If $\Sigma$ is a $k$-section of $(G,M)$, then $\tilde\Sigma=\{[eH,s]\mid s\in\Sigma\}$ is a $k$-section of $(G,M_\Sigma)$ and $W(\tilde\Sigma)=W(\Sigma)$. In particular, the foliation of $M_\Sigma$ given by the $G$-translates of $\tilde\Sigma$ 
    has totally geodesic leaves.
\item $(M_\Sigma)_{\tilde\Sigma}\simeq M_\Sigma$ ($G$-equivalent).
\item If $\Sigma$ is a minimal section of $(G,M)$, then $\copol(G,M_\Sigma)\le\copol(G,M)$.
\end{enumerate}
\end{proposition}
\begin{proof}
(i) is clear by the assumptions made on the metric on $G\!/\!H$.

(ii): By Theorem \ref{t:weylcover} (iii) we have that $\tilde\Sigma$ is complete, connected and embedded into $M_\Sigma$ and intersects every $G$-orbit. Consider the principal bundle
$$\psi:G\!/\!H\times\Sigma\to M_\Sigma,\ (gH,s)\mapsto [gH,s],$$
which maps a point $(gH,s)$ to its $W$-orbit $[gH,s]=\{(gn^{-1}H,n\cdot s)\mid nH\in W\}$.
By our choice of metric, $\psi$ is a Riemannian submersion.

We claim that $\tilde\Sigma$ is totally geodesic in $M_\Sigma$. In fact, $\psi^{-1}(\tilde\Sigma)=W\times\Sigma$ and since $W$ is totally geodesic in $G\!/\!H$ by Corollary \ref{c:G-W-invariant}, it follows that $W\times\Sigma$ is totally geodesic in $G\!/\!H\times\Sigma$. Thus $\tilde\Sigma=\psi(W\times\Sigma)$ is totally geodesic in $M_\Sigma$. This already yields properties \ref{pr:1} and \ref{pr:2} of a fat section. The fibre of $\psi$ over $[eH,s]$ is
$$\psi^{-1}([eH,s])=\{(nH,n^{-1}\cdot s)\mid nH\in W\}.$$
In order to speak about metric relations in the tangent spaces of $M_\Sigma$ we have to determine the vertical and horizontal distributions, $\kv$ and $\kh$, of $\psi$ along $\{eH\}\times\Sigma$.
$$\kv_{(eH,s)}:=T_{(eH,s)}\psi^{-1}([eH,s]) \text{ and } \kh_{(eH,s)}:=(\kv_{(eH,s)})^\bot.$$
The definition of the fibre yields
$$\kv_{(eH,s)}=\{(X+\gothh,-X_s)\mid X+\gothh\in\gothn/\gothh\}\subseteq\gothn/\gothh\times T_s(W\cdot s),$$
and a computation shows that
$$\kh_{(eH,s)}=((\gothn/\gothh)^\bot\times\nu^\Sigma_s(W\cdot s))\oplus A_s,$$
where $A_s:=\kh_{(eH,s)}\cap(\gothn/\gothh\times T_s(W\cdot s))$. In fact, $A_s$ corresponds to the tangent space of the $W$-orbit through $[eH,s]$ (induced by the left action of $G$) and one can show that
$$A_s=\{(f_s(v),v)\mid v\in T_s(W\cdot s)\},$$
for some linear monomorphism $f_s:T_s(W\cdot s)\to\gothn/\gothh$ (we do not need this fact in the following).
By our assumptions on the Riemannian metric on $G\!/\!H\times\Sigma$ and $M_\Sigma$, we have that $\psi$ is a Riemannian submersion. Hence, we may identify subspaces of $T_{[eH,s]}M_\Sigma$ with certain subspaces of $\kh_{(eH,s)}$. More precisely,
$$T_{[eH,s]}(G\cdot[eH,s])\simeq \kh_{(eH,s)}\cap(\underbrace{T_{(eH,s)}(G\cdot(eH,s))+\kv_{(eH,s)}}_{=\gothg/\gothh\times T_s(W\cdot s)})=A_s\oplus((\gothn/\gothh)^\bot\times\{0\}),$$
and it follows that
$$\nu_{[eH,s]}(G\cdot[eH,s])\simeq\{0\}\times\nu_s^\Sigma(W\cdot s)\subseteq(\{0\}\times\nu_s^\Sigma(W\cdot s))\oplus A_s\simeq T_{[eH,s]}\tilde\Sigma.$$
We therefore have for \emph{all} points $[eH,s]\in\tilde\Sigma$ (and not just the $G$-regular ones) that $$\nu_{[eH,s]}(G\cdot[eH,s])\subseteq T_{[eH,s]}\tilde\Sigma.$$
This shows property \ref{pr:3} of a fat section. We now come to property \ref{pr:4}.
If $[eH,s]\in\tilde\Sigma$ and $g\in G$ with $g\cdot[eH,s]=[gH,s]\in\tilde\Sigma$, it follows that $g\in N$ (again this holds not only in the $G$-regular points). We have therefore shown that $\tilde\Sigma$ is a $k$-section of $(G,M_\Sigma)$ if $\Sigma$ is a $k$-section of $(G,M)$. It is also not difficult to show $W(\tilde\Sigma)=W(\Sigma)$. In fact, $N_G(\tilde\Sigma)=N_G(\Sigma)$ and $Z_G(\tilde\Sigma)=Z_G(\Sigma)$.

(iii): Let $\tilde{\tilde\varphi}:(M_\Sigma)_{\tilde\Sigma}\to M_\Sigma$ denote the canonical $G$-equivariant surjection. That is
$$\tilde{\tilde\varphi}:G\!/\!H\times_W \tilde\Sigma\to M_\Sigma,\ [gH,[eH,s]]\mapsto [gH,s].$$
If $\tilde{\tilde\varphi}([gH,[eH,s]])=\tilde{\tilde\varphi}([\tilde gH,[eH,\tilde s]])$, then $[gH,s]=[\tilde gH,\tilde s]$. Now $\tilde gH=gn^{-1}H$ and $\tilde s=n\cdot s$ for some $n\in N$. But this implies
\begin{eqnarray*}
[\tilde gH,[eH,\tilde s]]&=&[gn^{-1}H,[eH,n\cdot s]]\\
&=&[gH,n\cdot[eH,n\cdot s]]\\
&=&[gH,[nH,n\cdot s]]\\
&=&[gH,[eH,s]].
\end{eqnarray*}
This shows that $\tilde{\tilde\varphi}$ is injective. By Theorem \ref{t:weylcover} (ii) we have $G_{[eH,s]}\subseteq N$ for all $s\in\Sigma$ and then (vi) of the same Theorem implies that $\tilde{\tilde\varphi}$ is a submersion. It follows that the map is a $G$-equivariant diffeomorphism.

(iv) is an immediate consequence of (ii).
\end{proof}
\pagebreak
\begin{remark}{\ }
\begin{enumerate}
\item
We do not know whether for a minimal section $\Sigma$ of $(G,M)$ it
is actually possible that $\copol(G,M_\Sigma)<\copol(G,M)$, or not.
\item
According to Proposition \ref{p:G-W-invariant} (iv), the assumptions
in the above Proposition above can be satisfied if $N$ is
compact. 

\item There are other natural ways to endow $M_\Sigma$ with a Riemannian metric such that
$\tilde\Sigma$ is totally geodesic, see for instance \cite[Theorem
9.59]{Bes}. We do not know if $G$ then still acts
isometrically on $M_\Sigma$ though.
\end{enumerate}
\end{remark}

The next result generalizes \cite[Proposition 2.6]{GS}, basically using the same proof.
\begin{proposition}\label{p:curvature}
Let $(G,M),\, G$ compact, be an isometric action with fat section $\Sigma$. If $\sec(M)\ge k$ for a $k\le 0$, then $\sec(M_\Sigma)\ge k$ for some Riemannian $G$-metric on $M_\Sigma$.
\end{proposition}

\begin{remark} As a concluding remark of this section, we show that for every triple $H\unlhd N\leq G$, where $G$ is a lie group, $H$ and $N$ are closed subgroups of $G$ and such that $N$ is compact, there exists some manifold $\Sigma$ on which $W=N\!/\!H$ acts isometrically, with trivial principal isotropy group and such that $M:=G\!/\!H\times_W\Sigma$ is a Riemannian $G$-manifold with fat section $\Sigma$ and fat Weyl group $W$. This generalizes the construction in \cite[5.6.20]{PT2}. In fact, since $W$ is compact, it acts faithfully on some Euclidean vector space $V$. Then $W$ acts with trivial principal isotropy group on the $k$-fold inner direct sum $\Sigma:=k\cdot V$ for some $k\in\N_{>0}$. If $G\!/\!H$ is endowed with a $(G$-$W)$-invariant Riemannian metric, then $M:=G\!/\!H\times_W\Sigma$ with the submersed metric from $G\!/\!H\times\Sigma$ is a $G$-manifold. Similarly as in the proof of Proposition \ref{p:resolution} (ii) one can show that $\tilde\Sigma:=\{[eH,s]\mid s\in\Sigma\}$ is a fat section with fat Weyl group $W$.
\end{remark}

\section{On a Generalization of Chevalley's Restriction Theorem}\label{s:chevalley}
Recall that a smooth $p$-form $\omega\in\Omega(M)$ is called
\textbf{$G$-invariant}, if for all $g\in G$ we have that
$g^*\omega=\omega$. The set of all $G$-invariant $p$-forms on $M$
will be denoted by $\Omega^p(M)^G$. A $p$-form $\omega$ is called
\textbf{horizontal}, if for all $X\in\gothg$ we have
$\iota_X(\omega)=0$. Here $\iota_X$ denotes contraction by the
Killing field generated by $X$. The set of all $G$-invariant
horizontal $p$-forms is denoted by $\Omega^p_\hor(M)^G$. These forms
are also called \textbf{basic} forms.

If $\Sigma$ is a fat section with fat Weyl group $W$, then in view
of Corollary \ref{c:c_isom} it is natural to ask whether the
isomorphism $\iota^*$ also yields
$\kc^\infty(M)^G\simeq\kc^\infty(\Sigma)^W$, or if we even have
$\Omega^*_\hor(M)^G\simeq\Omega^*_\hor(\Sigma)^W$. In the polar case
(i.e. $\copol(G,M)=0$) the first statement has been proved by Palais
and Terng in \cite{PT1} and the second statement by Michor in
\cite{M1,M2}. In the general case we note the following:
\begin{proposition}
The map $\iota^*:\kc^\infty(M)^G\to\kc^\infty(\Sigma)^W,\ f\mapsto
f|_\Sigma$, is well defined and injective, and the $G$-invariant
continuous extension $(\iota^*)^{-1}(f)$ of
$f\in\kc^\infty(\Sigma)^W$ to $M$ is smooth on $M^\reg$.
\end{proposition}
\begin{proof}
First note that $\iota^*|_{\kc^\infty(M)^G}$ is well defined,
because $\Sigma$ is an embedded submanifold of $M$. The injectivity
is also clear, since $\iota^*$ as a map on $\kc^0(M)^G$ is already
injective due to Corollary \ref{c:c_isom}. Let now
$f\in\kc^\infty(\Sigma)^W$ be arbitrary and denote its $G$-invariant
extension to $M$ by $F$. Smoothness of $F$ is a local condition.
Thus, let $p\in M$ be an arbitrary point and let $U$ be a tubular
neighborhood of $G\cdot p$. Since $F$ is $G$-invariant, we may
assume that $p\in\Sigma$. Let furthermore $S_p$ be a slice through p
such that $U=G\cdot S_p$. It is known that $F|_U$ is smooth in $p$
if and only if $F|_{S_p}$ is smooth in $p$. Since $\Sigma$ is a fat
section we have $S_p\subseteq \Sigma$ in the case that $p$ is a
$G$-regular point and $S_q$ is also a slice with respect to the
$W$-action on $\Sigma$. Hence $F|_{S_p}=f|_{S_p}$ is smooth in $p$.
\end{proof}
Suppose that for every $q\in\Sigma$ the relation $\nu_q(G\cdot
q)\subseteq T_q\Sigma$ holds. Then the arguments in the above proof
show that the $G$-invariant continuation of a smooth $W$-invariant
function on $\Sigma$ is smooth on the whole of $M$. In particular, if $M_\Sigma$ is the resolution of $(G,M)$ with respect to $\Sigma$ and $\tilde\Sigma=\{[gH,s]\mid s\in\Sigma\}$ is the fat section induced by $\Sigma$ (see Section \ref{s:resolution}), we have
\begin{corollary}
Let $H=Z_G(\Sigma)$. If $G/H$ carries a $(G$-$W)$-invariant
Riemannian metric, then
$\iota^*:\kc^\infty(M_\Sigma)^G\to\kc^\infty(\tilde\Sigma)^W$ is an
isomorphism.
\end{corollary}
\begin{proof} Due to Proposition \ref{p:resolution}, $\tilde\Sigma$
is a fat section of $M_\Sigma$. Let $q=[eH,s]\in\tilde\Sigma$ be an
arbitrary point. According to the Slice Theorem \ref{t:slicerep} the
set $V_q:=\nu_q(G\cdot q)\cap T_q\tilde\Sigma$ is a fat section of
the slice representation in $q$. Hence, for every $v\in\nu_q(G\cdot
q)$ there is some $g\in G_q$ with $v\in g\cdot V_q$. By Theorem
\ref{t:weylcover} $M_\Sigma$ is foliated by
$\{g\cdot\tilde\Sigma\mid g\in G\}$. Therefore,
$$g\cdot V_q=\nu_q(G\cdot q)\cap T_q(g\cdot\tilde\Sigma)=\nu_q(G\cdot q)\cap T_q(\tilde\Sigma)=V_q.$$
Thus $\nu_q(G\cdot q)\subseteq V_q\subseteq T_q\tilde\Sigma$ holds
for every $q\in\tilde\Sigma$.
\end{proof}

\begin{proposition}\label{p:basicform}
Let $(G,M)$ be an isometric action and let $\Sigma$ be a fat section
with fat Weyl group $W=W(\Sigma)$. Then the mapping
$i^*:\Omega^*_\hor(M)^G\to\Omega^*_\hor(\Sigma)^W$, which is
obtained by restriction to $\Sigma$, is injective.
\end{proposition}

\begin{proof}
The mapping $i^*$ is well defined, since $\Sigma$ is an embedded
submanifold and due to Corollary \ref{c:orbitparam}. Suppose now
that $i^*\omega=0$ for some $p$-form $\omega\in\Omega^*_\hor(M)^G$.
Let $q\in\Sigma\cap M^\reg$ be an arbitrary $G$-regular point in
$\Sigma$. By property \ref{pr:3} of a fat section, we have a (not
necessarily direct) decomposition of $T_q M=T_q \Sigma+T_q (G\cdot
q)$. Let $X_1,\dots, X_p$ be arbitrary vectors in $T_q M$. According
to the above decomposition we can write $X_i=Y_i+Z_i$, where $Y_i\in
T_q \Sigma$ and $Z_i\in T_q(G\cdot q)$ for all $i=1,\dots, p$. Now
$\omega_q(X_1,\dots, X_p)$ decomposes into a sum where each summand
contains either $Y_i$ or $Z_i$ for all $i=1,\dots, p$. If a summand
contains at least one $Z_i$, then it vanishes, since $\omega$ is
horizontal. Otherwise, the summand is $\omega_q(Y_1,\dots, Y_p)$ and
vanishes because $i^*\omega=0$. All in all we thus have that
$\omega_q=0$. Since $\omega$ is $G$-invariant, this holds along the
whole orbit through $q$. Now $q\in M^\reg$ was arbitrary, so
$\omega$ vanishes on the $G$-regular set of $M$, and since the
regular set is dense in $M$, we finally conclude that $\omega=0$ on
all of $M$.
\end{proof}
One would expect that $i^*$ should also be surjective in general.
However, we can show this only under strong assumptions:
\begin{theorem}\label{t:basicform}
Let $(G,M)$ be an isometric action and let $\Sigma\subseteq M$ be a
minimal section. Put $W=W(\Sigma)$. Suppose that the slice
representation $(G_q,\nu_q(G\cdot q))$ is polar for every
$q\in\Sigma$ and that $V_q=\nu_q(G\cdot q)\cap T_q\Sigma$ is a
$0$-section. Then $\Omega^*_\hor(M)^G\simeq\Omega^*_\hor(\Sigma)^W$.
In particular, $\kc^\infty(M)^G\simeq\kc^\infty(\Sigma)^W$ and the
isomorphism in both cases is given by the map $i^*$ from proposition
\ref{p:basicform}.
\end{theorem}

\begin{proof}
All that is left to show is the surjectivity of $i^*$. The proof is
basically the same as Michor's in \cite[4.2]{M1} (see also
\cite[Section 2.7]{Mag1}). Sketch of proof: Given a form
$\tilde\omega\in\Omega^p_\hor(\Sigma)^W$, we have to construct a
form $\omega\in\Omega^p_\hor(M)^G$ with $i^*(\omega)=\tilde\omega$.
In a first step, we locally construct $\omega$ using that the slice
representation is polar in every point and in combination with
\cite[Corollary 3.8 and Lemma 4.1]{M1}. The corollary states that
basic forms correspond to Weyl-invariant forms for polar
representations, and the lemma states that basic forms on a slice
can be extended to basic forms on the corresponding tube. Finally,
the various local forms are glued up via a $G$-invariant partition
of unity.
\end{proof}

\begin{remark}{\ }
\begin{enumerate}
\item
The assumption of polarity of the slice representation in the
theorem above enters in the step where \cite[Corollary 3.8 ]{M1} is
used. For this to work we need that for a polar representation
$(G,V)$ with $G$ compact and section $\Sigma$ with Weyl group $W$, the restriction
$p\mapsto p|_\Sigma$ induces an isomorphism $\R[V]^G\simeq\R[\Sigma]^W$.

\item For examples where the assumptions of Theorem \ref{t:basicform} hold see Theorem
\ref{t:ksection}.

\item The cohomology of the complex $\Omega^*_\hor(M)^G$
is \emph{basic cohomology} $H^*_{G-\basic}(M)$. Theorem \ref{t:basicform} implies $H^*_{G-\basic}(M)\simeq H^*_{W-\basic}(\Sigma)$, under the given assumptions.
Koszul observed in \cite{K} that for $M$ compact,
basic cohomology is isomorphic to the singular
cohomology of $G\backslash M$\begin{footnote}{I
thank Peter W. Michor for this information.}\end{footnote}. Hence,
using $G\backslash M\approx W\backslash\Sigma$ from Theorem \ref{t:orbitspaceisometry}, we obtain the
isomorphism of basic cohomology under the weaker assumption of $M$ being compact.
\end{enumerate}
\end{remark}

\section{Copolarity of Singular Riemannian Foliations}\label{s:srf}
Since pre-sections are purely geometrical objects and since minimal sections can be expressed as connected components of the intersections of certain pre-sections (see Proposition \ref{p:copolarproperties} (iv)), there is a meaningful way to define these notions for singular Riemannian foliations. This also leads to the notion of copolarity for the latter. A reference for the following notions is \cite[Chapter 6]{Mol}.
A \textbf{transnormal system} $\kf$ on a Riemannian manifold $M$ is a partition of $M$ into complete connected immersed submanifolds of $M$ such that every geodesic perpendicular to one leaf is perpendicular to all other leaves it meets. A \textbf{singular Riemannian foliation (SRF)} is a transnormal system such that the module $\Xi_\kf$ of all vector fields, which are tangent to all leaves in $\kf$, spans for every $p\in M$ the tangent space $T_pF$ of the leaf $F\in\kf$ through $p$. A leaf $F$ is called \textbf{regular} if it has maximal dimension, otherwise it is called \textbf{singular}.

The partition of a $G$-manifold $M$ into the $G$-orbits is a transnormal system. Since the tangent space of every orbit is spanned by the $G$-Killing fields, this partition is also an SRF. Note that principal and exceptional orbits are both considered as regular leaves for the singular foliation.

Pre-sections for an SRF with locally closed leaves can be defined as for $G$-manifolds:
\begin{definition}
Let $M$ be a Riemannian manifold and let $\kf$ be singular Riemannian foliation with locally closed leaves on $M$.
A submanifold $\Sigma\subseteq M$ is called a \textbf{pre-section} for $\kf$ if the following three conditions are satisfied:
\begin{enumerate}
\item $\Sigma$ is complete, connected, embedded and totally
geodesic in $M$,

\item $\Sigma$ intersects every leaf of $\kf$,

\item for every regular leaf $F\in\kf$ and all points $p\in\Sigma\cap F$ we have $\nu_p(F)\subseteq T_p\Sigma$.
\end{enumerate}
If $p\in M$ is a point which lies on a regular leaf, then a pre-section of least dimension which contains $p$ is called a \textbf{minimal section through $p$}.
\end{definition}
The properties of a singular Riemannian foliation together with the assumption that the leaves are locally closed in $M$ yield the following generalization of Lemma \ref{l:slice_intersect}:
\begin{lemma}
If $F\in\kf$ is an arbitrary leaf, then for every $q\in F$ the set $\exp_q(\nu_q(F))$ intersects any other leaf of $\kf$.
\end{lemma}
Let $p$ be a point in some regular leaf. Using the above lemma, it is easy to see that, if $\Sigma_1$ and $\Sigma_2$ are two pre-sections through $p$, then the connected component of $\Sigma_1\cap\Sigma_2$ which contains $p$ is again a pre-section. Hence, through every regular point $p$ passes a unique minimal section.
From here on it is seems quite natural to assume that all results on minimal sections of isometric group actions should carry over in one way or another to the case of minimal sections of SFRs with locally closed leaves. 

However, a noteworthy point is that a corresponding definition of canonical fat sections (Definition \ref{d:can_fat}) or cores (\cite{GS}) makes no sense for general singular Riemannian foliations with locally closed leaves. 
Hence, the minimal sections we defined above serve as a generalization of canonical fat sections.
\section[Copolarity of Actions induced by Polar Actions on Symmetric Spaces]{Copolarity of Actions induced by Polar Actions on Symmetric Spaces}\label{s:symspace}
In this section, our aim is to compute the copolarity of actions on compact lie groups which are associated to certain polar actions on symmetric spaces of compact type.

We first recall some notions for symmetric spaces in order to fix our notation (for the details we refer to Helgason's monograph \cite{H1}). A \emph{symmetric pair} $(G,K)$ consists of a Lie group $G$ and a closed subgroup $K$ such that an involutive automorphism $\sigma:G\to G$ exists with $\Fix(\sigma)^\circ\subseteq K\subseteq\Fix(\sigma)$. If in addition $\Ad_G(K)$ is compact, then the pair is called \emph{Riemannian}. 
The involution $\sigma$ induces an involution of the Lie algebra $\gothg$ of $G$ (also denoted by $\sigma$). This yields the so called Cartan-decomposition $\gothg=\gothk\oplus\gothp$, where $\gothk$ is the $(+1)$- and $\gothp$ the $(-1)$-eigenspace of $\sigma$. Note that $\gothk$ is at the same time the Lie algebra of $K$. If $\pi:G\to G\!/\!K$ denotes the canonical projection, then $T_{eK}G\!/\!K$ is identified with $\gothp$ via $d\pi(e)$.


It is well known that the complete connected totally geodesic submanifolds $\Sigma$ of $G\!/\!K$ correspond bijectively to the Lie triple systems $\gothm$ of $\gothp$. 
Furthermore, $\goths:=[\gothm,\gothm]\oplus\gothm$ is a Lie subalgebra of $\gothg$ and its corresponding Lie subgroup $S$ of $G$ together with $L:=S_{eK}=S\cap K$ form a Riemannian symmetric pair. 
We have $\Sigma=\pi(S)\simeq S\!/\!L$, and $S$ is the smallest subgroup of $G$ that acts transitively on $\Sigma$.

Although natural, we could not find a reference for the following statement.
\begin{lemma}\label{l:Sigma_embed}
Let $(G,K)$ be a Riemannian symmetric pair with $G$ compact. Suppose that $\Sigma\subseteq G\!/\!K$ is a complete, connected and totally geodesic submanifold. Then $\Sigma$ is embedded in $G\!/\!K$ if and only if $S$ is closed in $G$.
\end{lemma}
\begin{proof}
If $S$ is closed in $G$, then $S$ acts isometrically on $G\!/\!K$. Therefore, its orbit $S\cdot eK=\Sigma$ is an embedded submanifold of $G\!/\!K$.

Conversely, $\goths=[\gothm,\gothm]\oplus\gothm$ is a compact Lie algebra because $\gothg$ is. Let $\goths=\gothz(\goths)\oplus[\goths,\goths]$ denote the decomposition of $\goths$ into its center $\gothz(\goths)$ and its semisimple part $[\goths,\goths]$. It follows that $\exp([\goths,\goths])$ is closed in $G$ (\cite{Mo}, p. 615)  and hence compact. The same holds for
$$\exp([\gothm,\gothm])=(\exp([\goths,\goths])\cap \Fix(\sigma))^\circ.$$
Since $\Sigma$ is embedded in $G\!/\!K$, its image under 
$\phi: G\!/\!K\to G,\, gK\mapsto g\sigma(g)^{-1}$ yields the compact submanifold $\exp(\gothm)$ of $G$. 
Note that $\exp(\gothm)$ is closed under forming rational powers of elements. Applying $\sigma$ to an element of $\exp(\gothm)$ has the same effect as forming its inverse.
Clearly, $\exp(\gothm)$ projects onto $\Sigma$ under $\pi$.

We next claim that every element $s\in S$ can be written as a product $s=xy$ where $x\in\exp(\gothm)$ and $y\in\exp([\gothm,\gothm])$. In fact, let $s\in S$ be arbitrary and let $s_t$ be a path from $e$ to $s$. Let then $x_t$ be a path in $\exp(\gothm)$ which starts in $e$ and satisfies
$$x_t^2=s_t\sigma(s_t)^{-1}=\phi\circ\pi(s_t)\in\exp(\gothm)$$
for all $t$. We claim that $y_t:=x_t^{-1}s_t$ is a path in $S$, which is fixed by $\sigma$. In fact, 
\begin{eqnarray*}
\sigma(y_t)&=&\sigma(x_t^{-1})\sigma(s_t)=x_t\sigma(s_t)s_t^{-1}s_t\\
&=&x_t(x_t^2)^{-1}s_t=x_t^{-1}s_t=y_t.
\end{eqnarray*}
This shows $y\in\exp([\gothm,\gothm])$. It follows that $S=\exp(\gothm)\exp([\gothm,\gothm])$ is closed in $G$.
\end{proof}

Now let $G$ be a compact Lie group equipped with a bi-invariant metric. Viewed as a symmetric space, $G$ can be identified with $(G\!\times\! G)\!/\!\Delta(G)$, where $\Delta(G)=\{(g,g)\mid g\in G\}$. So $g\in  G$ is identified with the coset $[g,e]=\{(gh,h)\mid h\in G\}$. Let $N\subseteq G$ be a totally geodesic submanifold of $G$. Then $\gothn:=T_e N$ is a Lie triple system of $\gothg=L(G)$. 
As before, a transitive group of isometries of $N$ can be realized as a subgroup of $G\!\times\! G$: Let $\tilde\gothn:=\{(X,-X)\mid X\in \gothn\}\subset\gothg\times \gothg$. Obviously, $\tilde \gothn$ is a Lie triple system, hence we may consider the Lie subalgebra
$$\goths:=[\tilde\gothn,\tilde\gothn]\oplus\tilde\gothn=\triangle([\gothn,\gothn])\oplus\tilde\gothn=\langle([X,Y]+Z,[X,Y]-Z)\mid X,Y,Z \in\gothn\rangle.$$

\begin{lemma}\label{l:totallytangent}
Let $S\subseteq G\!\times\! G$ be the connected Lie subgroup of $G\!\times\! G$ with $L(S)=\goths$. Then $S$ is a group of isometries of $N$ and we have for all $(g,h)\in S:\ g\cdot N\cdot h^{-1}=N$ and therefore
$$T_{gh^{-1}}N=g\cdot T_e N\cdot h^{-1}=g\cdot \gothn\cdot h^{-1}.$$
In particular, $(\exp(X),\exp(-X))\in S$ for all $X\in \gothn$, and hence
$$T_{\exp(2X)}N=\exp(X)\cdot\gothn\cdot\exp(X).$$
\end{lemma}

Let now $(G,K)$ be a Riemannian symmetric pair with $G$ compact.
The reason for all the preliminary work is the following: Whenever $H$ is a subgroup of $G$, the action $\psi$ of $H$ on $G\!/\!K$ by left translation lifts to an action $\varphi$ of $H\!\times\! K$ on $G$ in the following way: $(h,k)\cdot g:= hgk^{-1}$. If $\pr_H:H\!\times\! K\to H$ denotes the projection onto the first factor, then the situation fits into the following commutative diagram:
$$\xymatrix{
{\ }\hspace{-7ex}(H\!\times\! K)\times G\ar@{->}[r]^-{\varphi}\ar@{->}_{\pr_H\times\pi}[d] &G\ar@{->}^\pi[d]\\
{\ }\hspace{1.5ex}H\times G\!/\!K\ar@{->}[r]_-{\psi} &G\!/\!K.
}$$
The lift $\varphi$ has certain distinctive features:
\begin{proposition}\label{p:HK_action}{\ }
\begin{enumerate}
\item $\pi$ maps $\varphi$-orbits onto $\psi$-orbits: $\pi(HgK)=H\cdot(gK)$. The orbit spaces\linebreak $(H\!\times\! K)\backslash G$ and $H\backslash G/\!K$ are canonically homeomorphic via $HgK\mapsto H\cdot(gK)$.

\item For the isotropy subgroups of both actions we have
\begin{eqnarray*}
(H\!\times\! K)_g&=&\{(h,g^{-1}hg)\mid h\in H\cap gKg^{-1}\} \text{ and}\\
H_{gK}&=&H\cap gKg^{-1}.
\end{eqnarray*}
Therefore, both groups are isomorphic via $\pr_H:(H\!\times\! K)_g\to H_{gK},\, (h,k)\mapsto h$.

\item The actions $\psi$ and $\varphi$ have the same cohomogeneity. More precisely, the slice of $\varphi$ through $g\in G$ is given by $\nu_g(HgK)=g\cdot(\Ad_{g^{-1}}(\gothh^\bot)\cap\gothk^\bot)$. The $\varphi$-orbits contain the fibres of $\pi$ and since they are mapped onto the orbits of $\psi$, the slice through $g\cdot p$ is given by $\nu_{gK}(H\cdot(gK))=d\pi(g)(\nu_g(HgK))$. Furthermore, the slice representation $((H\!\times\! K)_g,\, \nu_g(HgK))$ of $\varphi$ is equivariantly isomorphic to the slice representation $(H_{gK},\, \nu_{gK}(H\cdot (gK)))$ of $\psi$.
\end{enumerate}
\end{proposition}
For the details we refer to \cite{GorTh}.

\begin{theorem}\label{t:ksection}
Let $(G,K)$ be a Riemannian symmetric pair with compact $G$. Let $H$ be a closed subgroup of $G$.
If $(H,G\!/\!K)$ is polar and $\Sigma$ is a section through $eK$ with $\gothm:=T_{eK}\Sigma$, then
$$\copol(H\!\times\! K, G)=\dim([\gothm,\gothm]).$$
A minimal section through $e$ is given by the connected Lie subgroup $S$ corresponding to the Lie subalgebra $\goths:=[\gothm,\gothm]\oplus \gothm$.
\end{theorem}

\begin{proof}
We first show that $S$ contains a minimal section. In a second step we show that each minimal section contains $S$. Without loss of generality we may assume that $e$ is regular with respect to the $(H\!\times\! K)$-action.

Clearly, $S$ is totally geodesic and complete as it is a Lie
subgroup of $G$. Since $\Sigma$ is embedded in $G\!/\!K$, Lemma \ref{l:Sigma_embed} shows that $S$ is embedded in $G$. Furthermore, since $S$ maps under the projection
$\pi:G\to G\!/\!K$ onto $\Sigma$, it intersects every orbit. Now
suppose that $g\in S$ is regular with respect to the action $\varphi$. Then $\pi(g)=gK$ is regular with respect to $\psi$ and the normal space $\nu_g(HgK)$ to the orbit $HgK$ in $g$ is given
by
$$(\gothh^\bot\cdot g)\cap(g\cdot\gothp)=g\cdot(\Ad_{g^{-1}}(\gothh^\bot)\cap\gothp).$$
However,
since the $H$-action on $G\!/\!K$ is polar, we know that
$\Ad_{g^{-1}}(\gothh^\bot)\cap\gothp=\gothm$ (see \cite[p. 195]{Gor2}). Since $S$ is a Lie
subalgebra of $G$, its tangent space in $g$ is given by left
translation of $\goths$ with $g$, i.e. $T_gS=g\cdot \goths$.
Combining this with the above, we obtain:
$$\nu_g(HgK)=g\cdot(\Ad_{g^{-1}}(\gothh^\bot)\cap\gothp)=g\cdot \gothm\subseteq g\cdot \goths=T_gS.$$
We have therefore established that any minimal section is
contained in $S$.

Now assume that $N\subseteq S$ is a minimal section through $e$ and
write $\gothn:=T_eN$. In particular we have the inclusion
$\nu_g(HgK)=g\cdot \gothm\subseteq T_gN$ for all regular $g\in N$
and therefore $\gothm\subseteq\gothn$. Since the set of regular points
of the $H\!\times\! K$-action on $G$ is open and dense in $G$ and $e$
is assumed to be a regular point, there is a small $\varepsilon>0$, such that for all $t\in (-\varepsilon,\varepsilon)$ and
$X\in \gothm$ with unit length, the value of $g^2=\exp(t\cdot X)$
is regular. Applying the tangent space formula from lemma
\ref{l:totallytangent} it follows that $g^2\cdot \gothm\subseteq T_{g^2}N=g\cdot \gothn\cdot g$, or in other words:
$$\Ad_g(\gothm)=\Ad_{\exp(t\!/\!2\cdot X)}(\gothm)\subseteq\gothn.$$
Since $\Ad_{\exp(X)}=e^{\ad_X}$, it follows for all
$Y\in\gothm$ and $t\in\R$:
$$\Ad_{\exp(t\!/\!2\cdot X)}(Y)=e^{t\!/\!2\cdot\ad_X}(Y)\in\gothn.$$
Differentiating in $t=0$ yields that $\ad_XY=[X,Y]\in\gothn$. By
linearity of the Lie bracket we may thus conclude that
$[\gothm,\gothm]\subseteq \gothn$ and therefore
$\goths\subseteq\gothn$ which in turn implies $S\subseteq N$.
\end{proof}
\begin{remark}
We have proved along the lines that even if the action of $H$ on
$G\!/\!K$ is not polar, the following inequality still holds:
$$\copol(H\!\times\! K,G)\le\copol(H,G\!/\!K)+\dim([\gothm,\gothm]).$$
Here $\gothm$ is the tangent space of a minimal section through $eK$. To be more precise, if $\Sigma\subseteq G\!/\!K$ denotes a minimal section with respect to the action $\psi$ and $\gothm=T_{eK}\Sigma$, then $S:=\exp([\gothm,\gothm]\oplus\gothm)$ contains a minimal section of the action $\varphi$.
\end{remark}

\begin{corollary}
With the assumptions and notation as in Theorem \ref{t:ksection}:
\begin{enumerate}
\item Assuming that $\psi$ is polar, then $\varphi$ is polar if and only if it is hyperpolar.
\item If $H=\{e\}$, then $\copol(K,G)=\dim([\gothp,\gothp])$ (the action is by right translation), and the copolarity is trivial. 
\end{enumerate}
\end{corollary}

We can also describe the relation between the generalized Weyl group of $\Sigma$ and the fat Weyl group of $S$:
\begin{proposition}\label{p:normalizer}
In addition to the assumptions of Theorem \ref{t:ksection} let $e$ be regular.
\begin{enumerate}
\item $N_{H\!\times\! K}(S)=\{(h,k)\in H\!\times\! K\mid hk^{-1}\in S\}$ and $Z_{H\!\times\! K}(S)=\triangle(H\cap K)$.

\item $\pr_H(N_{H\!\times\! K}(S))=N_H(\Sigma)$ and $\pr_H(Z_{H\!\times\! K}(S))=Z_H(\Sigma)$.

\item The following diagram is commutative
$$\xymatrix{
N(S)\ar@{->>}[r]^{\pr_H}\ar@{->>}[d]_{p_1} & N(\Sigma)\ar@{->>}[d]^{p_2}\\
W(S)\ar@{->>}[r]_{\pr_W} & W(\Sigma),
}$$
where $\pr_W$ denotes the homomorphism induced by $p_2\circ\pr_H$. Hence, $W(S)$ is mapped canonically onto $W(\Sigma)$ and has at least as many connected components as the latter.

\item $N(\Sigma)\simeq N(S)/(\{e\}\!\times\!(K\cap S))$ and $W(\Sigma)\simeq W(S)/p_1(\{e\}\!\times\!(K\cap S))$.
\end{enumerate}
\end{proposition}
\begin{proof}
The description of the normalizer in (i) follows from property \ref{pr:4} of a fat section. The centralizer of a minimal section coincides with the isotropy group of any $(H\!\times\! K)$-regular point of $S$. Since $e$ is a regular point, $Z_{H\!\times\! K}(S)=\Delta(H\cap K)$ follows from Proposition \ref{p:HK_action} (ii).

Let $(h,k)\in N_{H\!\times\! K}(S)$ be arbitrary. If we apply $\pi$ to the equation $hSk^{-1}=S$ we obtain $h\cdot\Sigma=\Sigma$. This proves $h\in N_H(\Sigma)$. Conversely, assume that $h\in N_H(\Sigma)$ is an arbitrary element. In particular, $hK\in\Sigma$. Since $\pi(S)=\Sigma$, we can find an element $s\in S$ with $hK=sK$. It follows that $k:=s^{-1}h\in K$, which we rewrite as $hk^{-1}=s\in S$. Since $e$ is a regular point for the action $\varphi$, by assumption, we conclude that $(h,k)\in N_{H\!\times\! K}(S)$ by property \ref{pr:4} of a fat section. This completes the proof that $\pr_H$ maps $N_{H\!\times\! K}(S)$ onto $N_H(\Sigma)$.

The statement in (iii) is easily verified. The same is true in the case of (iv). In fact, the kernel of $\pr_H$ is given by
$$\ker(\pr_H)=\{(h,k)\in N(S)\mid h=1,\ k\in S\}=\{e\}\!\times\!(K\cap S).$$
\end{proof}
\begin{remark}
With the assumptions of Theorem \ref{t:ksection}, Proposition \ref{p:HK_action} (iii) shows that the assumptions made in Theorem \ref{t:basicform} are satisfied. I.e. the basic forms on $S$ and $G$ are naturally isomorphic to each other. In particular, the smooth $(H\!\times\! K)$-invariant functions on $G$ correspond to the smooth $N(S)$-invariant functions on $S$.
\end{remark}
The polar, non-hyperpolar actions on compact rank one symmetric
spaces yield interesting examples where Theorem \ref{t:ksection} is
applicable. These actions have been classified in \cite{PTh}. As an
example, consider the action of the torus $T^2\subset\Sun(3)$ on
$\mathbf{P}_2(\C)=\Sun(3)/S(\Un(1)\times\Un(2))$. It is polar with
Weyl group $\Z_2$, but not hyperpolar as it has $\mathbf{P}_2(\R)$
as a section. Its lift to the action of $T^2\times
S(\Un(1)\times\Un(2))$ on $\Sun(3)$ has nontrivial copolarity $1$
with minimal section given by $\Son(3)$. The fat Weyl group is
isomorphic to $\Z_2\times\Z_2\times\On(2)$ in this case.


\section{An Infinite Dimensional Isometric Action}\label{s:infinite}
In \cite{GOT} it is shown
that one may easily construct actions with prescribed fat sections
in the following way: Take a polar action $(G_1,M_1)$ with section
$\Sigma_1$ and any action $(G_2,M_2)$ whose principal orbit has
dimension $k$. Then
$$(G,M):=(G_1\!\times\! G_2,M_1\!\times\! M_2)$$
has $\Sigma:=\Sigma_1\!\times\! M_2$ as a $k$-section. If $(G_1,M_1)$ is an infinite dimensional isometric Fredholm action\begin{footnote}{A (proper) isometric action $(G,M)$ is called \textbf{Fredholm} if $\cohom(G,M)<\infty$.}\end{footnote} and $G_2$ and $M_2$ are finite dimensional, then it follows that $\Sigma_1\!\times\! M_2$ has finite dimension. Hence, $\copol(G,M)$ is also finite in this case. Besides these
constructed examples, one might ask if there exist isometric Fredholm actions of infinite
dimensional Lie groups on infinite dimensional manifolds with finite dimensional minimal sections. A natural candidate is
the action by gauge transformation, which we describe in the following (see \cite{PT2, TT1}). Let $G$ be a compact Lie group with a bi-invariant Riemannian metric and let $H$ and $K$ be closed subgroups of $G$. The action by \textbf{gauge transformation} is defined as:
$$*:\kp(G,H\!\times\! K)\times V\to V,\ (g,u)\mapsto \Ad_g(u)-g'g^{-1}=gug^{-1}-g'g^{-1}.$$
Here $\kp(G,H\!\times\! K)$ is the Hilbert-Lie group of $H^1$
paths $g:I\to G,\, (g(0),g(1))\in H\times K$, and let
$V=H^0(I;\gothg)=L^2(I;\gothg)$ be the Hilbert space of $L^2$
integrable paths $u:I\to\gothg$ in $\gothg=L(G)$, equipped with
the inner product
$$\langle u,v\rangle_0:=\int_0^1\langle u(t),v(t)\rangle_1\ dt\text{ with } \langle\cdot,\cdot\rangle_1\ \text{$\Ad_G$-invariant.}$$

We briefly summarize some facts concerning the gauge transformation without proofs:
\begin{enumerate}
\item $*$ is a smooth isometric Fredholm action by affine transformations.


\item The action of $\kp(G,e\!\times\! G)$ on $V$ is simply
transitive. In other words, the orbit map $\alpha:\kp(G,e\!\times\! G)\to V,\ g\mapsto g*\hat 0=-g'g^{-1}$ is a
diffeomorphism. 

\item The map $\phi:V\to G,\ u\mapsto \alpha^{-1}(u)(1)$, obtained by
mapping $u$ into $\kp(G,e\!\times\! G)$ and then evaluating in $t=1$, is a surjective Riemannian submersion.




\item The following diagram commutes:
$$\xymatrix{
{\ }\hspace{-11.5ex}\kp(G,H\!\times\! K)\times V\ar[rr]^-{*}\ar[d]_{\pi\times\phi} && V\ar[d]^\phi\\
{\ }\hspace{-7ex}(H\!\times\! K)\times G\ar[rr]_-{\varphi} && G,
}$$
where $\pi$
denotes the map $\pi:\kp(G,H\!\times\! K)\to H\!\times\! K,\ g\mapsto
(g(0),g(1))$. Thus, $\phi$ is equivariant with respect to $\pi$. Furthermore, the isotropy subgroups of both actions are isomorphic via $\pi$.

\item For $u\in V$ we have that $\kp(G,H\!\times\!
K)*u=\phi^{-1}((H\!\times\! K)\cdot\phi(u))$.

\item The fibres of $\phi$ coincide with the orbits of
$\Omega_e(G)=\kp(G,e\!\times\! e)$. That is
$$\phi^{-1}(\phi(u))=\Omega_e(G)*u,\ \text{for all } u\in V.$$
In particular, we have $\phi^{-1}(\exp(Y))=\Omega_e(G)*\hat Y$ for all
$Y\in\gothg$.

\item For $u\in V$ let $\tilde M:=\kp(G,H\!\times\! K)*u$. The tangent
space on $\tilde M$ in $u$ is:
$$T_u(\tilde M)=\{[\xi, u]-\xi'\mid \xi\in H^1(I;\gothg),\xi(0)\in\gothh,\xi(1)\in\gothk\}.$$

\item If $h\in \kp(G,H\!\times\! K)$ with $u=h*\hat 0$ and
$x=\phi(u)$, then:
$$\nu_u(\tilde M)=\{hbx^{-1}h^{-1}\mid b\in\nu_x(HxK)\}=\Ad_{hx}(\Ad_x^{-1}(\gothh^\bot) \cap\gothk^\bot).$$
Hence, $\nu_{\hat 0}(\tilde M)$ is the set of constant paths in $\Ad_x^{-1}(\gothh^\bot)\cap\gothk^\bot=\nu_x(HxK)$.
\end{enumerate}

Lemma \ref{l:slice_intersect} also holds for the action by gauge transformation:
\begin{lemma}
$\nu_{\hat 0}(\kp(G,H\!\times\! K)*\hat 0)$ intersects all orbits of $(\kp(G,H\!\times\! K),V)$.
\end{lemma}
\begin{proof} 
Let $\kp(G,H\!\times\! K)*u$ be an arbitrary orbit and put $x:=\phi(u)$.
Now consider $X\in\Ad_x^{-1}(\gothh^\bot)\cap\gothk^\bot=\nu_x(HxK)$ such that
$(H\!\times\! K)\cdot x=(H\!\times\! K)\cdot\exp(X)$.
Such an $X$ exists, because $\exp(\nu_x(HxK))$
intersect every $(H\!\times\!K)$-orbit on $G$. Using (v) above we obtain
$$\kp(G,H\!\times\! K)*u=\phi^{-1}((H\!\times\! K)\cdot\phi(u))=\phi^{-1}((H\!\times\! K)\cdot\exp(X)).$$
It follows, using (vi) above, that $\phi^{-1}(\exp(X))=\Omega_e(G)*\hat
X\subseteq\kp(G,H\!\times\! K)*u$.
\end{proof}

In the following, we assume that $(G,K)$ is a Riemannian symmetric pair with compact $G$ and that $H\subseteq G$ is a closed subgroup. As usual, we identify $T_{eK}G\!/\!K$ with $\gothp$ from the Cartan decomposition $\gothg=\gothk\oplus\gothp$. Our aim is to show that if $H$ acts polarly on $G\!/\!K$, then the action by gauge transformation is either polar (and hence hyperpolar), or it has infinite dimensional minimal sections and hence infinite copolarity. This gives a partial negative answer to the question we asked at the beginning of this section.

If $\gothl\subseteq\gothg$ is an arbitrary subset of $\gothg$, we
denote by $\hat\gothl\subseteq V$ the set of constant paths in
$V$ with value in $\gothl$. It is clear that if $\gothl$ is a
subspace (or subalgebra) of $\gothg$, then $\hat\gothl$ is a
subspace (resp. subalgebra) of $V$ which is canonically isomorphic
to $\gothl$. In particular, $\gothg$ is embedded into $V$ via
$\hat\gothg$.

\begin{lemma}\label{l:copolar}
Suppose that $H$ acts polarly on $G\!/\!K$ and let $\gothm\subseteq\gothp$ be the tangent space of a section through $eK$. If $eK$ is $H$-regular, then
every fat section $\ks\subseteq V$ of the $\kp(G,H\!\times\! K)$-action on $V$
through $\hat 0$ contains the linear subspace
$$\span\{t\mapsto e^{(1-t)\cdot\ad_X}(Y)\mid X\in \gothm \text{ regular},\ Y\in \gothm\}.$$
Here we call an element $X\in \gothg$ \textbf{regular}, if $\exp(X)\in G$
is regular with respect to the $(H\!\times\! K)$-action on $G$.
\end{lemma}
\begin{proof}
Since $\ks\subseteq V$ is supposed to be a fat section through $\hat 0$,
it is complete, connected, and totally geodesic in $V$. Hence,
$\ks$ has to be a linear subspace of $V$.

From $\hat\gothm=\nu_{\hat 0}(\kp(G,H\!\times\! K)*\hat 0)\subseteq T_{\hat 0}\ks=\ks$ and property \ref{pr:3} of a fat section we may conclude
that for all regular $\hat X\in\hat\gothm$ we have $\nu_{\hat
X}(\kp(G,H\!\times\! K)*\hat X)\subseteq \ks$. Let $h\in\kp(G,e\!\times\!
G)$ be the path defined by $h(t):=\exp(-t\cdot X)$. Then $X=h*\hat
0$ and $\phi(\hat X)=h(1)^{-1}=\exp(X)$ is a regular element for
the $(H\!\times\! K)$-action on $G$. Since the action of $H$ on $G\!/\!K$
is polar, it follows that $\Ad_{\exp(-X)}(\gothh^\bot)
\cap\gothp=\gothm$. From (viii) above we thus conclude that
$$\nu_{\hat X}(\kp(G,H\!\times\! K)*\hat X)=\Ad_{h\exp(X)}(\Ad_{\exp(-X)}(\gothh^\bot)\cap(\gothp))=\Ad_{h\exp(X)}(\gothm).$$
Since $h\exp(X)=\exp(-t\cdot X)\exp(X)=\exp((1-t)X)$ and
$\Ad_{\exp(X)}=e^{\ad_X}$, we obtain
$\Ad_{h\exp(X)}(\gothm)=\{t\mapsto e^{(1-t)\cdot\ad_X}(Y) \mid
Y\in \gothm\}$. This fact together with $\ks$ being linear
completes the proof.
\end{proof}

\begin{theorem}
Let $(G,K)$ be a Riemannian symmetric pair with compact $G$ and let $H\subseteq G$ be a
closed subgroup. Supposed that the action of
$H$ on $G\!/\!K$ is polar, then the following are equivalent:
\begin{enumerate}
\item $\copol(\kp(G,H\!\times\! K),V)<\infty$.

\item $\copol(\kp(G,H\!\times\! K),V)=0$.

\item The action of $\kp(G,H\!\times\! K)$ on $V$ is hyperpolar.

\item The action of $H\!\times\! K$ on $G$ is hyperpolar.

\item The action of $H$ on $G\!/\!K$ is hyperpolar.
\end{enumerate}
\end{theorem}
\begin{proof}
The equivalence of (iii), (iv) and (v) is well known.
Furthermore, since sections in $V$ are automatically flat and
$\copol=0$ implies that an action is polar, (iii) is equivalent to
(ii). Certainly, (ii) implies (i).

Let $\Sigma$ be a section of $(H,G\!/\!K)$ through $eK$ and assume that $eK$ is $H$-regular. Then $e$ is $H\!\times\! K$-regular and $\hat 0$ is regular with respect to the action by gauge transformation. Put $\gothm:=T_{eK}\Sigma$. We now show that if $\copol(\kp(G,H\!\times\! K),V)\ne 0$ then the copolarity must already be infinite.
Let $X,Y\in\gothm$ be elements with
$$(\ad_X)^2(Y)=-\delta Y\ne 0 \text{ and } \|Y\|=1.$$
Such elements exist, since $\gothm$ is a
Lie triple system and $\gothm$ is not Abelian. Otherwise, $\hat\gothm$ would
be a $0$-section, which contradicts our assumption. Recalling that
$0$ is regular, there is a ball of regular elements around $0\in
\gothm$. In fact, this is clear since $e\in G$ is regular and the
set of regular points is open in $G$. We may thus further assume
that $\varepsilon\cdot X$ is regular for all $\varepsilon\in [0,1]$. By lemma \ref{l:copolar}, every minimal section $\ks$ of $(\kp(G,H\!\times\! K),V)$ contains the
infinite family
$$\km:=\{t\mapsto e^{(1-t)\!/\!p_n\cdot\ad_X}(Y)\mid n\in \textbf{N}\}\subseteq V,$$
where $p_n$ denotes the $n$-th odd prime number. We claim that
$\km$ is linearly independent. Since every equivalence class in
$L^2(I;\gothg)$ has at most one continuous representative, it suffices to
show that the family $\km$ is linearly independent as a subset of
$\kc(I;\gothg)$. Furthermore, all members of $\km$ are analytic maps
which can be extended analytically to $\textbf{R}$ and so we only
need to show that they are linearly independent when viewed as
functions on $\textbf{R}$.

Now assume there exist $\lambda_1,\dots,\lambda_n\in \textbf{R}$
such that
$$t\mapsto \sum_{k=1}^n \lambda_k e^{(1-t)\!/\!p_k\cdot \ad_X}(Y)=0.$$
For any $s:=(1-t)\in \textbf{R}$, we then have:
\begin{eqnarray*}
0&=&\left\langle \sum_{k=1}^n \lambda_k e^{s\!/\!p_k\cdot
\ad_X}(Y),Y\right\rangle=
\sum_{k=1}^n \lambda_k \langle e^{s\!/\!p_k\cdot \ad_X}(Y),Y\rangle\\
 &=& \sum_{k=1}^n \lambda_k \sum_{l=0}^\infty\frac{s^l}{p_k^l\cdot l!}
 \langle (\ad_X)^l(Y),Y\rangle\\
 &=&\sum_{k=1}^n \lambda_k \sum_{l=0}^\infty\frac{s^{2l}}{p_k^{2l}\cdot
 (2l)!} \delta^{2l}(-1)^l\\
 &=&\sum_{k=1}^n \lambda_k \cos\left(\frac{s\delta}{p_k}\right),
\end{eqnarray*}
where we made use of the continuity of the $\Ad$-invariant inner
product $\langle\cdot,\cdot\rangle$ and the fact that
$$\langle
(\ad_X)^l(Y),Y\rangle=\left\{
\begin{array}{cl}
0 & \text{, for $l$ odd}\\
(-1)^{l\!/\!2}\delta^l & \text{, for $l$ even.}
\end{array}\right.$$
By choosing
$$s_k:=\pi\cdot (\prod_{l=1}^np_l)\!/\!(2 \delta p_k)$$
for $k=1,\dots, n$, we obtain that $\lambda_k=0$. 
We have thus shown, that for any $n\in
\textbf{N}$ the first $n$ members of $\km$ are linearly
independent, which completes our proof.
\end{proof}


\section{Appendix - Invariant Metrics}\label{s:invariant}
We are interested in left-$G$-invariant metrics on a homogeneous space $G\!/\!H$ which are also right-invariant under a certain group $W$. This concept generalizes that of a $G$-invariant metric on $G\!/\!H$ and is used in Section \ref{s:resolution}. First recall that any triple $(H\unlhd N\leq G)$, where $G$ is a Lie group, $H$ and $N$ are closed subgroups of $G$ and $H$ is normal in $N$, gives rise to a $W$-principal bundle:
$$W\hookrightarrow G\!/\!H\twoheadrightarrow G\!/\!N,$$
where $W=N\!/\!H$. In this situation $G$ acts on $G\!/\!H$ from the left and $W$ acts properly and freely on $G\!/\!H$ from the right by $(gH,nH)\mapsto gnH$. We are interested in the case that these actions are isometric, so we are lead to consider Riemannian metrics on $G\!/\!H$ which are left-$G$- and right-$W$-invariant.

\begin{definition}
A Riemannian metric on $G\!/\!H$ which is both left-$G$- and right-$W$-invariant is called \textbf{$(G$-$W)$-invariant}.
\end{definition}

\begin{proposition}\label{p:G-W-invariant}{\ }
\begin{enumerate}
\item The $(G$-$W)$-invariant Riemannian metrics on $G\!/\!H$ are in $1-1$ correspondence with the $\Ad_G(N)$-invariant scalar products on $\gothg/\gothh$.
\item If $N$ is connected, then a scalar product $\langle\cdot|\cdot\rangle$ on $\gothg/\gothh$ is $\Ad_G(N)$-invariant if and only if $\ad_\gothn$ is skew-symmetric with respect to $\langle\cdot|\cdot\rangle$.
\item If $\gothg/\gothh$ admits a decomposition $\gothg/\gothh=\gothn/\gothh\oplus\gothp$ with $\Ad_G(N)(\gothp)\subseteq\gothp$, then the $\Ad_G(N)$-invariant scalar products on $\gothg/\gothh$, which satisfy 
$(\gothn/\gothh)\bot\gothp$, are in $1-1$ correspondence with pairs $(\langle\cdot|\cdot\rangle_{\gothn/\gothh},\langle\cdot|\cdot\rangle_\gothp)$ 
of $\Ad_W$-invariant scalar products on $\gothn/\gothh$, resp. 
$\Ad_G(N)$-invariant scalar products on $\gothp$.

Such a pair exists if and only if $W$ is covered by a product of a compact Lie group with a vector group and if the image of $N$ under $n\mapsto\Ad_n|_\gothp$ in $\Gln(\gothp)$ is relatively compact.

Conversely, if $\langle\cdot|\cdot\rangle$ on $\gothg/\gothh$ is $\Ad_G(N)$-invariant, then $\langle\cdot|\cdot\rangle|_{\gothn/\gothh}$ is $\Ad_W$-in\-vari\-ant. 
If $\gothp:=(\gothn/\gothh)^\bot$, then $\Ad_G(N)(\gothp)\subseteq\gothp$ and $\langle\cdot|\cdot\rangle|_{\gothp}$ is $\Ad_G(N)$-invariant. 
\item If $N$ is compact, then $G\!/\!H$ admits a $(G$-$W)$-invariant Riemannian metric.
\end{enumerate}
\end{proposition}
\begin{proof}
(i): Let $h$ be a Riemannian metric on $G\!/\!H$ and put $\langle\cdot|\cdot\rangle:=h_{eH}$. Then it is well known that $h$ is left-$G$-invariant if and only if $\langle\cdot|\cdot\rangle$ is $\Ad_G(H)$-invariant. If additionally $h$ is right-$W$-invariant, then $r_n^*h=h$ for all $n\in N$. Hence, we have for all $X,Y\in T_{gH}G\!/\!H$:
$$h_{gnH}(X\cdot n,Y\cdot n)=h_{gH}(X,Y).$$
Using the $G$-invariance we obtain
$$h_{eH}(n^{-1}g^{-1}\cdot X\cdot n,n^{-1}g^{-1}\cdot X\cdot n)=h_{eH}(g^{-1}\cdot X,g^{-1}\cdot Y).$$
Under the natural identification $T_{eH} G\!/\!H\simeq\gothg/\gothh$ this is equivalent to
$$\langle \Ad_n(X)\mid\Ad_n(Y)\rangle=\langle X \mid Y\rangle, \text{ for all } X,Y\in\gothg/\gothh.$$
Conversely, if we are given an $\Ad_G(N)$-invariant scalar product on $\gothg/\gothh$ then, in particular, it is $\Ad_G(H)$-invariant. It therefore gives rise to  a left-$G$-invariant Riemannian metric on $G\!/\!H$. Furthermore, it is easy to see, that it is right-$W$-invariant.

(ii): This is a standard consideration.

(iii): If $\langle\cdot|\cdot\rangle$ is an $\Ad_G(N)$-invariant scalar product on $\gothg/\gothh$ satisfying $(\gothn/\gothh)\bot\gothp$, then its restriction to $\gothn/\gothh$ resp. $\gothp$ clearly yields the stated pair of invariant scalar products. Conversely, we may patch such a pair of invariant scalar products together to form an $\Ad_G(N)$-invariant scalar product $\langle\cdot|\cdot\rangle$ on $\gothg/\gothh$ by defining:
$$\langle X+Y| Z+W\rangle:=\langle X|Z\rangle_{\gothn/\gothh}+\langle Y| W\rangle_\gothp,\text{ for all } X,Z\in\gothn/\gothh,\, Y,W\in\gothp.$$
The $\Ad_W$-invariance of $\langle\cdot|\cdot\rangle_{\gothn/\gothh}$ is equivalent to the existence of a bi-invariant Riemannian metric on $W$. Using \cite[Proposition 3.34]{CE} yields that this is the case if and only if $W$ is covered by the product of a compact Lie group and a vector group. Also, if $\langle\cdot|\cdot\rangle_\gothp$ is $\Ad_G(N)$-invariant, then the image of $N$ under $f:N\to\Gln(\gothp),\, n\mapsto\Ad_n|_\gothp$ is contained in the compact group $\On(\gothp)$ and therefore relatively compact. Conversely, if $K:=\overline{f(N)}\subseteq\Gln(\gothp)$ is compact, then we may define by an averaging process a $K$-invariant scalar product on $\gothp$, which in turn is $\Ad_G(N)$-invariant.

(iv) follows from (iii) and the fact that a representation of a compact Lie group is completely reducible.
\end{proof}
The following Proposition shows that the concept of a $(G$-$W)$-invariant metric on $G\!/\!H$ is actually the same as that of a left-$G$-invariant metric on $G\!/\!N$. 

\begin{proposition}
The $(G$-$W)$-invariant metrics on $G\!/\!H$ correspond to the left-$G$-invariant metrics on $G\!/\!N$.
\end{proposition}
\begin{proof}
If we are given a $(G$-$W)$-invariant metric on $G\!/\!H$, then the submersed metric on $G\!/\!N$ under the canonical $G$-equivariant mapping $gH\mapsto gN$ is left-$G$-invariant. Conversely, if we start with a left-$G$-invariant metric on $G\!/\!N$, then $G$ admits a left invariant metric which is right-$N$-invariant. This induces an $\Ad_G(N)$-invariant scalar product on $\gothg$ and since $\gothh$ is $\Ad_G(N)$-invariant, the induced scalar product on $\gothg/\gothh$ is $\Ad_G(N)$-invariant. Using Proposition \ref{p:G-W-invariant} (i), this yields a $(G$-$W)$-invariant Riemannian metric on $G\!/\!H$.
\end{proof}

The next result is basically \cite[Theorem 9.80]{Bes}.
\begin{corollary}\label{c:G-W-invariant}
If $G\!/\!H$ carries a $(G$-$W)$-invariant Riemannian metric, then the principal fibre bundle $G\!/\!H\twoheadrightarrow G\!/\!N$ is a Riemannian submersion, where $G\!/\!N$ is endowed with the quotient metric. Its fibres are totally geodesic. In particular $W$, viewed as a subset of $G\!/\!H$, is totally geodesic in $G\!/\!H$. Furthermore, the map
$$(\gothn/\gothh)^\bot\to\gothg/\gothn,\ X+\gothh\mapsto X+\gothn$$
is a linear isometry.
\end{corollary}
\begin{proof}
By left-$G$-invariance, the fibres of the principal bundle $G\!/\!H\twoheadrightarrow G\!/\!N$ are all isometric to the fibre $W$ over $eN$. Now $W$ is the image of $N$ under the canonical projection $G\to G\!/\!H$, which is a Riemannian submersion if $G$ is endowed with a left-invariant metric that is right-$N$-invariant. Such a metric exists due to, \cite[Proposition 3.16]{CE}. By the following lemma, $N$ is a totally geodesic submanifold of $G$. Hence, its image $W$ under the Riemannian submersion $G\to G\!/\!H$ is totally geodesic in $G\!/\!H$.
\end{proof}

\begin{lemma}
Let $G$ be a Lie group and $H\subseteq G$ a closed subgroup. If $G$ carries a left-invariant metric which is right $H$-invariant, then the induced metric on $H$ is bi-invariant and $H$ is a totally geodesic submanifold of $G$.
\end{lemma}
\bibliography{Literaturverzeichnis}
\bibliographystyle{alpha}
\end{document}